\documentclass[12pt,a4paper]{amsart}
\usepackage{graphicx}
\usepackage[ansinew]{inputenc}    
\usepackage[all]{xy}

\usepackage{amssymb,amscd,empheq}
\usepackage{amsmath,amsfonts}
\usepackage{amsthm}

\usepackage[toc,titletoc]{appendix}

\usepackage[normalem]{ulem} 

\usepackage{hyperref}

\usepackage{tikz}

\newtheorem{theorem}{Theorem}[section]
\newtheorem{definition}[theorem]{Definition}
\newtheorem{proposition}[theorem]{Proposition}
\newtheorem{lemma}[theorem]{Lemma}
\newtheorem{corollary}[theorem]{Corollary}
\newtheorem{remark}[theorem]{Remark}

\def\C{\mathbb C}
\def\R{\mathbb R}

\def\N{\mathbb N}

\usepackage{graphicx}
\usepackage[ansinew]{inputenc}    
\usepackage[all]{xy}
\usepackage{amssymb,amscd}
\usepackage{color}

\numberwithin{equation}{section}

\title[Lower semicontinuity of attractors for a coupled system]{Lower semicontinuity of pullback attractors for a non-autonomous coupled system of strongly damped wave equations}

\author[E. M. Bonotto]{Everaldo M. Bonotto$^{\dag}$}\thanks{$\dag$Research partially
supported by FAPESP \#2020/14075-6
 and CNPq  \#  310540/2019-4}
\address[E. M. Bonotto]{Instituto de Ci\^encias Matem\'{a}ticas e de Computa\c{c}\~ao, Universidade de S\~{a}o Paulo, Campus de S\~{a}o Carlos, Caixa Postal 668, 13560-970, S\~{a}o Carlos SP, Brazil.}
\email{ebonotto@icmc.usp.br}

\author[A. N. Carvalho]{Alexandre N. Carvalho$^{\dag\dag}$}\thanks{$\dag\dag$Research partially
supported by FAPESP \#2020/14075-6}
\address[A. N. Carvalho]{Instituto de Ci\^encias Matem\'{a}ticas e de Computa\c{c}\~ao, Universidade de S\~{a}o Paulo, Campus de S\~{a}o Carlos, Caixa Postal 668, 13560-970, S\~{a}o Carlos SP, Brazil.}
\email{andcarva@icmc.usp.br}

\author[M. J. D. Nascimento]{Marcelo J. D. Nascimento$^\star$}\thanks{$^\star$Research partially supported by CNPq \# 302743/2022-7 and FAPESP \# 2020/14075-6.}

\address[M. J. D. Nascimento]{Universidade Federal de S\~{a}o
Carlos, Departamento de Matem\'atica, 13565-905, S\~{a}o
Carlos SP, Brazil.}
\email{marcelojdn@ufscar.br}

\author[E. B. Santiago]{Eric B. Santiago} 
\address[E. B. Santiago]{Instituto de Ci\^encias Matem\'{a}ticas e de Computa\c{c}\~ao, Universidade de S\~{a}o Paulo, Campus de S\~{a}o Carlos, Caixa Postal 668, 13560-970, S\~{a}o Carlos SP, Brazil.}
\email{ericbsantiago@icmc.usp.br}

\date{}

\begin{document}

\begin{abstract}
The aim of this paper is to study the robustness of the family of pullback attractors associated to a non-autonomous coupled system of strongly damped wave equations given by the following evolution system 
\[
\begin{cases}
u_{tt} - \Delta u + u + \eta(-\Delta)^{1\slash 2}u_t + a_{\epsilon}(t)(-\Delta)^{1\slash 2}v_t = f(u), & (x, t) \in\Omega\times (\tau, \infty), \\
v_{tt} - \Delta v + \eta(-\Delta)^{1\slash 2}v_t - a_{\epsilon}(t)(-\Delta)^{1\slash 2}u_t = 0, & (x, t) \in\Omega\times (\tau, \infty),
\end{cases}
\]
subject to boundary conditions
\[
u = v = 0, \;  (x, t) \in\partial\Omega\times (\tau, \infty), 
\]
and initial conditions
\[
u(\tau, x) = u_0(x), \ u_t(\tau, x) = u_1(x), \ v(\tau, x) = v_0(x), \ v_t(\tau, x) = v_1(x), \ x \in \Omega, \ \tau\in\mathbb{R}, 
\]
where $\Omega$ is a bounded smooth domain in $\mathbb{R}^N$, $N \geq 3$, with the boundary $\partial\Omega$ assumed to be regular enough, $\eta > 0$ is a constant, $a_{\epsilon}$ is a H\"{o}lder continuous function satisfying uniform boundedness conditions, and $f\in C^1(\mathbb{R})$ is a dissipative nonlinearity with subcritical growth. 
This problem is a modified version of the well known Klein-Gordon-Zakharov system. 
Under suitable hyperbolicity conditions, we obtain the gradient-like structure of the limit pullback attractor associated with this evolution system, and we prove the continuity of the family of pullback attractors at $\epsilon = 0$.




\end{abstract}

\maketitle

\section{Introduction}

The continuity of attractors with respect to perturbations is a growing subject that has called the attention of many mathematicians over the years. 
In the literature, the study of the robustness of attractors is usually divided into two main features:  \textit{upper semicontinuity} and  \textit{lower semicontinuity}. 

Roughly speaking, upper semicontinuity ensures that the unperturbed attractor remains stable when the perturbation is well-behaved. On the other hand, lower semicontinuity means that the unperturbed attractor does not undergo a collapse or degeneration. The first property is typically expected due to the convergence of the perturbed dynamical system under a small perturbation. However, obtaining a result on lower semicontinuity is considerably more challenging, as it necessitates specific structural assumptions about the dynamics within the unperturbed attractor. Moreover, characterizing the structure of attractors is a crucial step towards achieving a result on lower semicontinuity.

In the case of an autonomous dynamical system (nonlinear semigroup) subject to an autonomous perturbation, only a structural assumption for the limit dynamical system is necessary. This means that only the limit global attractor needs to be gradient, and all the equilibrium points of the limit semigroup are hyperbolic. For further details, refer to  \cite{Hale and Raugel}, where the authors established an abstract result and also explored applications to partial differential equations. In other words, no additional characterization was needed for the perturbed problem.
 
Later, in \cite{Carvalho and Langa}, the authors explored a scenario in which the perturbed dynamical system is non-autonomous, while the associated limit dynamical system is autonomous. In this context, the attractor is defined as the union of unstable manifolds around hyperbolic equilibrium points. Furthermore, through the introduction of a non-autonomous term into an autonomous equation, they proved how the concept of a hyperbolic equilibrium point (in the autonomous case) turns to the concept of a hyperbolic global solution for evolution processes (within the non-autonomous framework). Additionally, they established the continuity of local unstable and stable manifolds in relation to perturbations.

For non-autonomous problems under regular perturbation, the work \cite{Carvalho Ergodic}  focuses on the property of exponential dichotomy for non-autonomous linear equations. It provides more general results that enable the treatment of cases where the unperturbed pullback attractor is given as the closure of the union of unstable manifolds around an infinite collection of hyperbolic time-dependent global solutions. 
 
Regarding applications of the abstract results in the literature, we refer to \cite{Rivero1}, where the authors obtained the continuity of pullback attractors at $\epsilon = 0$ for a non-autonomous strongly damped wave equation, given by the following formulation 
\[
u_{tt} - \Delta u - \gamma(t) \Delta u_t + \beta_{\epsilon}(t) u_t = f(u),
\]
in a bounded domain $\Omega\subset \mathbb{R}^N$, where the time-dependent terms $\gamma(t)$ and $\beta_{\epsilon}(t)$ are continuously differentiable in $\mathbb{R}$ satisfying some uniform boundedness conditions, and the nonlinear term $f$ has appropriate growth restrictions.

Meanwhile, in \cite{Parreira}, the author studied the lower semicontinuity of the family of pullback attractors for the following singularly non-autonomous plate equation 
\[
\begin{cases}
u_{tt} + a_{\epsilon}(t, x)u_t + (-\Delta)u_t + (-\Delta)^2u + \lambda u = f(u), & \mbox{in} \quad \Omega, \\
u = \Delta u = 0, & \mbox{on} \quad \partial\Omega, 
\end{cases}
\]
with respect to the damping coefficient $a_{\epsilon}$, where $\Omega$ is a bounded domain in $\mathbb{R}^N$, $\lambda >0$, and $f$ is a suitable nonlinearity. This type of problem is related to models of vibrations in elastic systems.

In this paper, we aim to investigate the robustness of the family of pullback attractors associated with a non-autonomous coupled system of strongly damped wave equations, as examined in \cite{BNS}. This system is defined by the following evolution equations:
\begin{equation}\label{intro system}
\begin{cases}
u_{tt} - \Delta u + u + \eta(-\Delta)^{1\slash 2}u_t + a_{\epsilon}(t)(-\Delta)^{1\slash 2}v_t = f(u), & (x, t) \in\Omega\times (\tau, \infty), \\
v_{tt} - \Delta v + \eta(-\Delta)^{1\slash 2}v_t - a_{\epsilon}(t)(-\Delta)^{1\slash 2}u_t = 0, & (x, t) \in\Omega\times (\tau, \infty), \\
u = v = 0, & (x, t) \in\partial\Omega\times (\tau, \infty), 
\end{cases}
\end{equation}
with initial conditions
\[
u(\tau, x) = u_0(x), \ u_t(\tau, x) = u_1(x), \ v(\tau, x) = v_0(x), \ v_t(\tau, x) = v_1(x), \ x \in \Omega, \ \tau\in\mathbb{R}, 
\]
where $\Omega$ is a bounded smooth domain in $\mathbb{R}^N$, $N \geq 3$, with the boundary $\partial\Omega$ assumed to be regular enough, $\eta > 0$ is a constant, $a_{\epsilon}$ is a H\"{o}lder continuous function satisfying uniform boundedness conditions, and $f\in C^1(\mathbb{R})$ is a dissipative nonlinearity with a subcritical growth restriction. 

The problem \eqref{intro system} is a non-autonomous modified version of the so called \textit{Klein-Gordon-Zakharov} (KGZ) system. 
In the autonomous case with space dimension $N=3$, it is well known that the KGZ system arises to describe the interaction between a Langmuir wave and an ion acoustic wave in a plasma, see \cite{Bellan, Dency, OT} and the references therein. 

In \cite{BNS}, the authors used the theory of uniform sectorial operators to establish the local and global well-posedness of the problem \eqref{intro system} in the product space $H_0^1(\Omega) \times L^2(\Omega) \times H_0^1(\Omega) \times L^2(\Omega)$. 
By proving that the nonlinear evolution process, associated to the global solution of \eqref{intro system}, is a compact map, it was established the existence of pullback attractors for the problem \eqref{intro system}. 
Using energy estimates and progressive increases of regularity, the authors were able to improve the regularity of the pullback attractors. Lastly, under a suitable assumption on the functional parameter $a_{\epsilon}(t)$, they obtained the upper semicontinuity of the family of pullback attractors at $\epsilon = 0$.

In this work, we study the limit system ($\epsilon = 0$) associated with \eqref{intro system} in order to obtain a characterization for the unperturbed pullback attractor in terms of the unstable manifolds around equilibrium points, assuming that all of them are hyperbolic for the limit problem. With the aid of a Lyapunov Functional, we prove that any bounded global solution in the limit pullback attractor is forward and backward asymptotic to equilibria. 
We also study the linearized system associated with \eqref{intro system} in order to prove a result on exponential dichotomy under perturbation, see Proposition \ref{convergence processes linearization} and Corollary \ref{all hyperbolic}. 
Moreover, thanks to abstract general results on lower semicontinuity from \cite{Rivero1, Carvalho and Langa, Carvalho Ergodic}, we obtain the lower semicontinuity of pullback attractors at $\epsilon = 0$. Consequently, we achieve the continuity of the family of pullback attractors at $\epsilon = 0$, concluding in this way our study of robustness for the problem \eqref{intro system}. 

 It is important to mention that, in our system \eqref{intro system}, the coupling terms 
\[
a_{\epsilon}(t)(-\Delta)^{1\slash 2}v_t \quad \mbox{and} \quad 
- a_{\epsilon}(t)(-\Delta)^{1\slash 2}u_t
\]
have the same order, in the sense of fractional powers, as well as the strong damping terms $\eta(-\Delta)^{1\slash 2}u_t$ and 
$\eta(-\Delta)^{1\slash 2}v_t$. 
This scenario generated a situation where we were not able to apply \cite[Theorem 7.6.11]{Henry} directly in order to obtain our result on exponential dichotomy for the problem \eqref{intro system}. 
For more details, see Section \ref{sub exponential dichotomy}, in particular Lemma \ref{conv semigroups in Y0 Y1}, Proposition \ref{convergence processes linearization} and Corollary \ref{all hyperbolic}, where we had to carry out an analysis that is not straightforward, and then we apply \cite[Theorem 7.6.10]{Henry} to conclude our proof on exponential dichotomy. In fact, there were several difficulties that we had to overcome, and to the best of our knowledge, the analysis presented in this work is novel up to the present moment. The reader may compare our results with those presented in \cite{Rivero1} and \cite{Parreira}, for instance. It is worth mentioning that if we replace the coupling terms $a_{\epsilon}(t)(-\Delta)^{1\slash 2}v_t$ and 
$- a_{\epsilon}(t)(-\Delta)^{1\slash 2}u_t$ in system \eqref{intro system}, respectively, by 
$a_{\epsilon}(t)v_t$ and $- a_{\epsilon}(t)u_t$, then one may proceed exactly as in \cite{Rivero1} in order to obtain the continuity of the family of pullback attractors at $\epsilon = 0$.

In what follows, we describe the outline of this paper. 

\begin{enumerate}

\item[$\bullet$] Section \ref{Abstract setting} is devoted to formulate the non-autonomous problem \eqref{intro system} in an abstract form  within the framework of nonlinear evolution processes. We present sufficient assumptions that ensure global well-posedness and the existence of pullback attractors for the problem \eqref{intro system} in an appropriate energy space. Additionally, we recall some auxiliary results established in \cite{BNS}. 

\item[$\bullet$] Section \ref{limit problem} deals with the structure of the limit problem (that is, with parameter $\epsilon = 0$) associated to \eqref{intro system}. In Proposition \ref{prop de Lyapunov}, we exhibit a Lyapunov functional associated with system \eqref{intro system}. In Theorem \ref{structure limit PA}, we establish the gradient-like structure of the limit pullback attractor, that is, it can be written as the union of the unstable manifolds around equilibria. 

\item[$\bullet$] Section \ref{linearized system} encompasses a spectral analysis of the matrix operator associated to the linearization of problem \eqref{intro system} around an equilibrium point. Here, we prove a continuity result for the inverse operator concerning the parameter $\epsilon$ (see Proposition \ref{convergence inverses}). Furthermore, we derive sectorial estimates that are uniform in both time and parameter $\epsilon$. Using these uniform estimates, we prove the convergence of linear semigroups, uniformly on compact subsets, as $\epsilon$ approaches to zero (see Proposition \ref{convergence linear semigroups}).

\item[$\bullet$] Section \ref{sub exponential dichotomy} presents the results regarding the robustness of problem \eqref{intro system}. Through a combination of arguments and results from \cite{Nascimento} and \cite{Henry}, we establish a result on exponential dichotomy under perturbation (with respect to parameter $\epsilon$), see Proposition \ref{convergence processes linearization} and Corollary \ref{all hyperbolic}. The lower semicontinuity of the family of pullback attractors at $\epsilon = 0$ is obtained in Theorem \ref{lower PA}.

\item[$\bullet$] Appendix \ref{appendix A} is dedicated to presenting a collection of preliminary facts that are useful for the development of the article. We provide definitions and results concerning the theory of nonlinear evolution processes and pullback attractors.

\item[$\bullet$] Appendix \ref{appendix B} addresses an abstract result on lower semicontinuity of pullback attractors which will be used in the main result.

\end{enumerate}


\section{Abstract setting of the problem}\label{Abstract setting}

Consider the following initial boundary value problem 
\begin{equation}\label{edp01}
\begin{cases}
u_{tt} - \Delta u + u + \eta(-\Delta)^{1\slash 2}u_t + a_{\epsilon}(t)(-\Delta)^{1\slash 2}v_t = f(u), & (x, t) \in\Omega\times (\tau, \infty), \\
v_{tt} - \Delta v + \eta(-\Delta)^{1\slash 2}v_t - a_{\epsilon}(t)(-\Delta)^{1\slash 2}u_t = 0, & (x, t) \in\Omega\times (\tau, \infty),
\end{cases}
\end{equation}
subject to boundary conditions
\begin{equation*}
u = v = 0, \;  (x, t) \in\partial\Omega\times (\tau, \infty), 
\end{equation*}
and initial conditions
\begin{equation}\label{cond01}
u(\tau, x) = u_0(x), \ u_t(\tau, x) = u_1(x), \ v(\tau, x) = v_0(x), \ v_t(\tau, x) = v_1(x), \ x \in \Omega, \ \tau\in\mathbb{R}.
\end{equation}

In what follows, we recall the sufficient conditions that ensure the global well-posedness of the non-autonomous problem $(\ref{edp01})$-$(\ref{cond01})$, which is established in \cite{BNS}. 

At first place, the non-autonomous problem $(\ref{edp01})$-$(\ref{cond01})$ can be formulated in a nonlinear evolution process setting by introducing the following notations. Let $X = L^2(\Omega)$ and denote by $A\colon  D(A)\subset X\to X$ the negative Laplacian operator with the Dirichlet boundary conditions, that is, $Au = (-\Delta)u$ for all $u \in D(A)$, where $D(A) = H^2(\Omega) \cap H_0^1(\Omega)$. It is well-known that $A$ is a positive self-adjoint operator in $X$ with compact resolvent and, following Henry \cite{Henry}, $A$ is a sectorial operator in $X$. Thus, $-A$ generates a compact analytic semigroup on $X$. Now, denote by $X^{\alpha}$, $\alpha > 0$, the fractional power spaces associated with the operator $A$, that is, $X^{\alpha} = D(A^{\alpha})$ endowed with the graph norm. With this notation, we have $X^{-\alpha} = (X^{\alpha})^{\prime}$ for all $\alpha>0$, see \cite{Amann}.

With this setting, the non-autonomous problem $(\ref{edp01})$-$(\ref{cond01})$ can be rewritten as an ordinary differential equation in the following abstract form
\begin{equation}\label{edp abstrata} 
\begin{cases}
W_t + \mathcal{A}_{\epsilon}(t)W = F(W), & \ t > \tau, \\
W(\tau) = W_0, & \ \tau \in \mathbb{R},
\end{cases}
\end{equation}
where $W = W(t)$, for all $t \in \mathbb{R}$, and $W_0 = W(\tau)$ are respectively given by
\begin{equation*}
W = \begin{bmatrix} u\\ u_t\\ v\\ v_t\\ \end{bmatrix} \ \text{and} \ \ W_0 = \begin{bmatrix} u_0\\ u_1\\ v_0\\ v_1\\ \end{bmatrix},
\end{equation*}
and, for each $t\in\mathbb{R}$, the unbounded linear operator $\mathcal{A}(t)\colon D(\mathcal{A}(t)) \subset Y_0 \to Y_0$ is defined by
\begin{equation}\label{abstract operator}
\begin{split}
\mathcal{A}_{\epsilon}(t)\!\! \begin{bmatrix} u\\ v\\ w\\ z\\ \end{bmatrix} \!\!&=\!\!
\begin{bmatrix}
0     & -I                           & 0 & 0 \\
A+I & \eta A^{1\slash 2} & 0 & a_{\epsilon}(t) A^{1\slash 2} \\
0     & 0                            & 0 & -I \\
0     & - a_{\epsilon}(t) A^{1\slash 2} & A & \eta A^{1\slash 2}
\end{bmatrix}
\!\!
\begin{bmatrix} u\\ v\\ w\\ z\\ \end{bmatrix} \\
\end{split}
\end{equation}
for each $\begin{bmatrix} u & v & w & z\end{bmatrix}^T$ in the domain $D(\mathcal{A}_{\epsilon}(t))$ defined by
\begin{equation}\label{domain abstract operator}
D(\mathcal{A}_{\epsilon}(t)) = X^1 \times X^{1\slash 2} \times X^1 \times X^{1\slash 2},
\end{equation}
where $Y_0 = X^{1\slash 2} \times X \times X^{1\slash 2} \times X$ 
is the phase space of the problem $(\ref{edp01})$-$(\ref{cond01})$. The nonlinearity $F$ is given by
\begin{equation}\label{nonlinearity F}
F(W) = \begin{bmatrix} 0\\ f^e(u)\\ 0\\ 0\\ \end{bmatrix},
\end{equation}
where $f^e(u)$ is the Nemitski\u i operator associated with $f(u)$, that is,
$f^e(u)(x) = f(u(x)),$ for all $x\in\Omega.$

In order to ensure the global well-posedness of the problem $(\ref{edp01})$-$(\ref{cond01})$ in $Y_0$, some suitable assumptions on the functional parameter $a_{\epsilon}$ and on the nonlinearity $f$ are needed. Assume that the function $a_{\epsilon}\colon  \mathbb{R} \to (0, \infty)$ is continuously differentiable in $\mathbb{R}$ and satisfies the following condition:
\begin{equation}\label{function a is bounded}
0 < a_0 \leq a_{\epsilon}(t) \leq a_1,
\end{equation}
for all $\epsilon \in [0,1]$ and $t\in \mathbb{R}$, with positive constants $a_0$ and $a_1$, and we also assume that the first derivative of $a_{\epsilon}$ is uniformly bounded in $t$ and $\epsilon$, that is, there exists a constant $b_0 > 0$ such that
\begin{equation*}
\vert a_{\epsilon}^{\prime}(t)\vert \leq b_0 \quad \mbox{for all} \quad t\in\mathbb{R}, \ \epsilon\in [0, 1].
\end{equation*}
The functional parameter $a_{\epsilon}$ will be assumed to be $(\beta, C)$-H\"{o}lder continuous for each $\epsilon \in [0,1]$, that is,
\begin{equation}\label{hol-a}
\vert a_{\epsilon}(t) - a_{\epsilon}(s)\vert\leq C\vert t-s\vert^{\beta},
\end{equation}
for all $t, s\in \mathbb{R}$ and $\epsilon \in [0, 1],$ where the constants $C > 0$ and $0 < \beta \leq 1$ are both independent of the time $t\in\mathbb{R}$ and also of the parameter $\epsilon$. Additionally, we will assume that
\begin{equation}\label{norm function a}
\Vert a_{\epsilon} - a_0 \Vert_{L^{\infty}(\mathbb{R})} \rightarrow 0 \quad \mbox{as} \quad \epsilon\rightarrow 0^{+}.
\end{equation}

Moreover, suppose that $f\in C^1(\mathbb{R})$ and it satisfies the dissipativeness condition
\begin{equation}\label{dissipativeness}
\limsup\limits_{\vert s\vert \to\infty} \frac{f(s)}{s} \leq 0,
\end{equation}
and also satisfies the subcritical growth condition
\begin{equation}\label{Gcondition}
\vert f^{\prime}(s)\vert \leq c(1 + \vert s\vert^{\rho - 1}), \quad \text{for all} \;\; s\in \mathbb{R},
\end{equation}
where $1 < \rho < \frac{N+2}{N-2}$, with $N \geq 3$, and $c>0$ is a constant. For a detailed study of equations of type \eqref{edp abstrata} in the parabolic case, we refer to \cite{Nascimento,Sobo}.

 Now, let $\alpha\in (0, 1)$ be such that 
\begin{equation}\label{alpha condition}
1 < \rho < \frac{N + 2(1-\alpha)}{N-2},
\end{equation}
 with $N \geq 3$. Recall that $Y_{\alpha - 1} = [Y_{-1}, Y_0]_{\alpha} $, where
$$
Y_{-1} = X \times X^{-1\slash 2} \times X \times X^{-1\slash 2}
$$
is the extrapolation space of $Y_0$, and $[\cdot, \cdot]_{\alpha}$ denotes the complex interpolation functor (see \cite{Triebel}).

Under the previous conditions, the following result on global well-posedness holds, see \cite[Theorem 2.1]{BNS}. 

\begin{theorem}\label{global-sol} \rm \textbf{[Global Well-Posedness]} \it \, 
Let $1 < \rho < \frac{N + 2}{N-2}$, $N\geq 3$, $\alpha \in (0, 1)$ satisfying condition \eqref{alpha condition}, and let $f \in C^1(\mathbb{R})$ be a function satisfying \eqref{dissipativeness}-\eqref{Gcondition}. 
Assume that conditions \eqref{function a is bounded}-\eqref{hol-a} hold and let $F\colon Y_0 \to Y_{\alpha - 1}$ be defined as in \eqref{nonlinearity F}. Then for any initial data $W_0 \in Y_0$ the problem \eqref{edp abstrata} has a unique global solution $W(t)$ such that $W(t) \in C([\tau,\infty), Y_0).$ 
Moreover, such solutions are continuous with respect to the initial data on $Y_0$.
\end{theorem}

Since the problem \eqref{edp01}-\eqref{cond01} has a global solution $W(t)$ in $Y_0$, we can define an evolution process $\{S(t, \tau)\colon t\geq\tau\in\mathbb{R}\}$ in $Y_0$ by
\begin{equation}\label{evolution process of the problem}
S(t, \tau)W_0 = W(t) =  (u, u_t, v, v_t), \quad  t\geq\tau\in\mathbb{R}.
\end{equation}

See \cite[Theorem 5.1]{BNS} and \cite[Proposition 5.4]{BNS} for a proof of Theorems \ref{the solution is exponentially dominated} and  \ref{compact S} below.

\begin{theorem}\label{the solution is exponentially dominated}
There exists $R > 0$ such that for any bounded subset $B \subset Y_0$ one can find $t_0(B) > 0$ satisfying
$\Vert S(t, \tau)W_0 \Vert_{Y_0}^2 \leq R$ for all $t \geq  \tau + t_0(B)$ and $W_0 \in B$. 
In particular, the evolution process $\{S(t, \tau)\colon  t\geq\tau\in\mathbb{R}\}$ defined in \eqref{evolution process of the problem} is pullback strongly bounded dissipative.
\end{theorem}

\begin{theorem}\label{compact S} For each $t > \tau \in \R$, the evolution process $S(t, \tau)\colon Y_0 \to Y_0$ given in \eqref{evolution process of the problem} is  a compact map.
\end{theorem}

In \cite[Theorem 2.2]{BNS} it was established that, for each $\epsilon\in [0, 1],$ the nonlinear evolution process $\{S_{\epsilon}(t, \tau)\colon t\geq\tau\in\mathbb{R}\}$, associated to the problem \eqref{edp01}-\eqref{cond01}, has a pullback attractor $\{ \mathbb{A}_{\epsilon}(t)\colon t\in\mathbb{R} \}$ in the phase space $Y_0$.

\begin{theorem}\label{ex pullback} \rm \textbf{[Existence of Pullback Attractors]} \it
Under the conditions of Theorem \ref{global-sol}, for each $\epsilon\in [0, 1],$ the problem \eqref{edp01}-\eqref{cond01} has a pullback attractor 
$\{ \mathbb{A}_{\epsilon}(t)\colon t\in\mathbb{R} \}$ in $Y_0$ and, moreover, the union
$\bigcup\limits_{t \in \mathbb{R}} \mathbb{A}_{\epsilon}(t) \subset Y_0$
is bounded.
\end{theorem}

The next result states the regularity of the pullback attractors obtained in Theorem \ref{ex pullback}, see \cite[Theorem 2.3]{BNS} for a proof.

\begin{theorem}\label{regularity pullback attractor} \rm \textbf{[Regularity of Pullback Attractors]} \it 
Assume that $\frac{N-1}{N-2} \leq \rho < \frac{N}{N-2}$, $N\geq 3$. For each $\epsilon\in [0, 1],$ the pullback attractor $\{ \mathbb{A}_{\epsilon}(t)\colon t\in\mathbb{R}\}$ for the problem \eqref{edp01}-\eqref{cond01}, obtained in Theorem \ref{ex pullback}, lies in a more regular space than $Y_0$. More precisely, we have $\bigcup\limits_{t \in \mathbb{R}} \mathbb{A}_{\epsilon}(t)$ is a bounded subset of 
$
Y_1 = X^1 \times X^{1\slash 2} \times X^1 \times X^{1\slash 2}.
$
\end{theorem}

As a consequence of Theorem \ref{regularity pullback attractor},  it follows by the compact embedding $Y_1 \hookrightarrow Y_0$  that
\begin{equation}\label{compactness limit set}
\overline{\bigcup\limits_{\epsilon\in [0, 1]} \bigcup\limits_{t \in \mathbb{R}} 
\mathbb{A}_{\epsilon}(t)}
\mbox{\,\, is compact in \,} Y_0.
\end{equation}


Lemma \ref{Aux applic L 0} deals with the continuity of the evolution process $\{S_{\epsilon}(t, \tau): t \geq \tau \in \R\}$  as $\epsilon \to 0^+$. This result is proved in \cite[Theorem 2.4]{BNS}.

\begin{lemma}\label{Aux applic L 0}
Let $\{S_{\epsilon}(t, \tau)\colon t \geq \tau \in \R\}$ be the evolution processes 
associated to the problem \eqref{edp01}-\eqref{cond01}. Then $S_{\epsilon}(t, \tau)x \xrightarrow{\epsilon \rightarrow 0^+} S_0(t, \tau)x$ uniformly in bounded subsets of $\{(t, \tau) \in \mathbb{R}^2 \colon t \geq \tau\} \times Y_0$.
\end{lemma}

In \cite[Theorem 2.4]{BNS} it was also proved the upper semicontinuity of the family of pullback attractors associated to the problem \eqref{edp01}-\eqref{cond01}. The next theorem states this result.


\begin{theorem}\label{upper PA} \rm \textbf{[Upper Semicontinuity]} \it 
Under condition \eqref{norm function a}, the family of pullback attractors 
$\{\mathbb{A}_{\epsilon}(t)\colon t\in\mathbb{R}\}$, associated to the problem \eqref{edp01}-\eqref{cond01}, is upper semicontinuous at $\epsilon = 0$, that is, for each $t\in\mathbb{R}$, 
$$
{\rm d_H} (\mathbb{A}_{\epsilon}(t), \mathbb{A}_0(t)) \rightarrow 0 
\quad \mbox{as} \quad \epsilon\rightarrow 0^{+}.
$$
\end{theorem}

\section{Structure of the limit problem}\label{limit problem}

In this section, we are going to study the structure of the limit problem
\begin{equation}\label{limit system}
\begin{cases}
u_{tt} - \Delta u + u + \eta(-\Delta)^{1\slash 2}u_t + a_0(t)(-\Delta)^{1\slash 2}v_t = f(u), & (x, t) \in\Omega\times (\tau, \infty), \\
v_{tt} - \Delta v + \eta(-\Delta)^{1\slash 2}v_t - a_0(t)(-\Delta)^{1\slash 2}u_t = 0, & (x, t) \in\Omega\times (\tau, \infty), \\
u = v = 0, & (x, t) \in\partial\Omega\times (\tau, \infty), \\
(u(\tau), u_t(\tau), v(\tau), v_t(\tau)) = (u_0, u_1, v_0, v_1),
\end{cases}
\end{equation}
which is associated to the problem \eqref{edp01}-\eqref{cond01} with parameter $\epsilon = 0$. Firstly, we will prove a result on asymptotic stability for solutions of \eqref{limit system} starting in the pullback attractor. Thereafter, this will be used to show that the limit pullback attractor $\{\mathbb{A}_0(t)\colon t\in\mathbb{R}\}$ can be written as the union of the unstable manifolds around equilibria.

 From now on, we shall assume that $\frac{N-1}{N-2} \leq \rho < \frac{N}{N-2}$. Recall from Theorem \ref{regularity pullback attractor} that this condition ensures that $\overline{\bigcup\limits_{\epsilon\in [0, 1]} \bigcup\limits_{t \in \mathbb{R}} 
\mathbb{A}_{\epsilon}(t)}$ is compact in $Y_0$.

\subsection{Asymptotic stability of solutions} Recall that $a_{\epsilon}$ is $(\beta, C)$-H\"{o}lder continuous for each $\epsilon \in [0,1]$, with $0 < \beta \leq 1$ (see \eqref{hol-a}). By Arzel\`a-Ascoli Theorem, for each fixed $t\in\mathbb{R}$, given a sequence $\{t_n\}_{n\in\mathbb{N}} \subset\mathbb{R}$, we have
\begin{equation}\label{a_n def}
\{a_n(t) = a_0(t_n + t)\}_{n\in\mathbb{N}}
\end{equation}
has a convergent subsequence (which we denote by the same notation) such that $$a_n(t) \xrightarrow{n\rightarrow\infty} a_{\infty}(t)$$ uniformly for $t$ in compact subsets of $\mathbb{R}$. Note that the limit function $a_{\infty}$ is also bounded as in \eqref{function a is bounded}. Thus, we may consider the following problems
\begin{equation}\label{problem n}
\begin{cases}
u_{tt} - \Delta u + u + \eta(-\Delta)^{1\slash 2}u_t + a_n(t)(-\Delta)^{1\slash 2}v_t = f(u), & (x, t) \in\Omega\times (\tau, \infty), \\
v_{tt} - \Delta v + \eta(-\Delta)^{1\slash 2}v_t - a_n(t)(-\Delta)^{1\slash 2}u_t = 0, & (x, t) \in\Omega\times (\tau, \infty), \\
u = v = 0, & (x, t) \in\partial\Omega\times (\tau, \infty), \\
(u(\tau), u_t(\tau), v(\tau), v_t(\tau)) = (u_0, u_1, v_0, v_1) \in Y_0,
\end{cases}
\end{equation}
and
\begin{equation}\label{problem with limit function}
\begin{cases}
u_{tt} - \Delta u + u + \eta(-\Delta)^{1\slash 2}u_t + a_{\infty}(t)(-\Delta)^{1\slash 2}v_t = f(u), & (x, t) \in\Omega\times (\tau, \infty), \\
v_{tt} - \Delta v + \eta(-\Delta)^{1\slash 2}v_t - a_{\infty}(t)(-\Delta)^{1\slash 2}u_t = 0, & (x, t) \in\Omega\times (\tau, \infty), \\
u = v = 0, & (x, t) \in\partial\Omega\times (\tau, \infty), \\
(u(\tau), u_t(\tau), v(\tau), v_t(\tau)) = (u_0, u_1, v_0, v_1) \in Y_0,
\end{cases}
\end{equation}
with global solutions given, respectively, by $W^n(t) = (u^n(t), u_t^n(t), v^n(t), v_t^n(t))$ and $W^{\infty}(t) = (u^{\infty}(t), u_t^{\infty}(t), v^{\infty}(t), v_t^{\infty}(t))$, $t\geq \tau\in \mathbb{R}.$

Moreover, for $n\in\mathbb{N}$, let
\[
\{\mathbb{A}_n(t)\colon t\in\mathbb{R}\} 
\quad \mbox{and} \quad 
\{\mathbb{A}_{\infty}(t)\colon t\in\mathbb{R}\}
\]
be the pullback attractors associated with the evolution problems \eqref{problem n} and \eqref{problem with limit function}, respectively.

\begin{theorem}\label{stability}
If $W^n(\cdot) = (u^n(\cdot), u_t^n(\cdot), v^n(\cdot), v_t^n(\cdot))$ and $W^{\infty}(\cdot) = (u^{\infty}(\cdot), u_t^{\infty}(\cdot), v^{\infty}(\cdot), v_t^{\infty}(\cdot))$ are the global solutions associated to the problems \eqref{problem n} and \eqref{problem with limit function}, respectively, starting in their pullback attractors, then
\[
\Vert W^n(t) - W^{\infty}(t) \Vert_{Y_0}^2 \xrightarrow{n\rightarrow\infty} 0
\]
in compact subsets of $\mathbb{R}$.
\end{theorem}
\begin{proof}
If we denote $u = u^n - u^{\infty}$ and $v = v^n - v^{\infty}$, then from systems \eqref{problem n} and \eqref{problem with limit function}, we are able to obtain the following system
\[
\begin{cases}
u_{tt} - \Delta u + u + \eta(-\Delta)^{1\slash 2}u_t + a_n(t)(-\Delta)^{1\slash 2}v_t^n - 
a_{\infty}(t)(-\Delta)^{1\slash 2}v_t^{\infty} = f(u^n) - f(u^{\infty}), \\
v_{tt} - \Delta v + \eta(-\Delta)^{1\slash 2}v_t - a_n(t)(-\Delta)^{1\slash 2}u_t^n + 
a_{\infty}(t)(-\Delta)^{1\slash 2}u_t^{\infty} = 0,
\end{cases}
\]
for all $t>\tau$ and $x\in\Omega$, which can be rewritten in the following way
\[
\begin{cases}
u_{tt} - \Delta u + u + \eta(-\Delta)^{1\slash 2}u_t + a_n(t)(-\Delta)^{1\slash 2}v_t + 
[a_n(t) - a_{\infty}(t)](-\Delta)^{1\slash 2}v_t^{\infty} = f(u^n) - f(u^{\infty}), \\
v_{tt} - \Delta v + \eta(-\Delta)^{1\slash 2}v_t - a_n(t)(-\Delta)^{1\slash 2}u_t + 
[a_{\infty}(t) - a_n(t)](-\Delta)^{1\slash 2}u_t^{\infty} = 0,
\end{cases}
\]
for all $t>\tau$ and $x\in\Omega$. Taking the inner product of the first equation in the previous system with $u_t$ in $X$, and also the inner product of the second equation with $v_t$ in $X$, we get
\[
\begin{split}
&\frac{1}{2}\frac{d}{dt}\int_{\Omega}\vert u_t\vert^2 \, dx 
+ \frac{1}{2}\frac{d}{dt}\int_{\Omega}\vert\nabla u\vert^2 \, dx 
+ \frac{1}{2}\frac{d}{dt}\int_{\Omega}\vert u\vert^2 \, dx 
+ \eta\Vert (-\Delta)^{1\slash 4}u_t \Vert_X^2 \\
&+ a_n(t) \langle (-\Delta)^{1\slash 2}v_t, u_t \rangle_X 
+ [a_n(t) - a_{\infty}(t)] \langle (-\Delta)^{1\slash 2}v_t^{\infty}, u_t \rangle_X \\
&= \langle f(u^n) - f(u^{\infty}), u_t \rangle_X
\end{split}
\]
and
\[
\begin{split}
&\frac{1}{2}\frac{d}{dt}\int_{\Omega}\vert v_t\vert^2 \, dx 
+ \frac{1}{2}\frac{d}{dt}\int_{\Omega}\vert\nabla v\vert^2 \, dx 
+ \eta\Vert (-\Delta)^{1\slash 4}v_t \Vert_X^2 \\
&- a_n(t) \langle (-\Delta)^{1\slash 2}u_t, v_t \rangle_X 
+ [a_{\infty}(t) - a_n(t)] \langle (-\Delta)^{1\slash 2}u_t^{\infty}, v_t \rangle_X = 0,
\end{split}
\]
respectively. Adding up these two last equations, it follows that
\[
\begin{split}
&\frac{1}{2}\frac{d}{dt} \left( \int_{\Omega}\vert\nabla u\vert^2 \, dx + 
\int_{\Omega}\vert u\vert^2 \, dx + 
\int_{\Omega}\vert u_t\vert^2 \, dx + 
\int_{\Omega}\vert\nabla v\vert^2 \, dx + 
\int_{\Omega}\vert v_t\vert^2 \, dx \right) \\
&\quad + \eta\Vert (-\Delta)^{1\slash 4} u_t \Vert_X^2 + 
\eta\Vert (-\Delta)^{1\slash 4} v_t \Vert_X^2 + 
[a_n(t) - a_{\infty}(t)] \langle (-\Delta)^{1\slash 2}v_t^{\infty}, u_t \rangle_X \\
&\quad + [a_{\infty}(t) - a_n(t)] \langle (-\Delta)^{1\slash 2}u_t^{\infty}, v_t \rangle_X \\
&= \int_{\Omega} [f(u^n) - f(u^{\infty})] u_t \, dx,
\end{split}
\]
which leads to
\[
\begin{split}
&\frac{d}{dt}\left( \Vert u\Vert_{X^{1\slash 2}}^2 + \Vert u\Vert_X^2 + \Vert u_t\Vert_X^2 + 
\Vert v\Vert_{X^{1\slash 2}}^2 + \Vert v_t\Vert_X^2 \right) \\
&= -2\eta\left( \Vert (-\Delta)^{1\slash 4} u_t \Vert_X^2 + 
\Vert (-\Delta)^{1\slash 4} v_t \Vert_X^2 \right) + 
2[a_{\infty}(t) - a_n(t)] \langle (-\Delta)^{1\slash 2}v_t^{\infty}, u_t \rangle_X \\
&\quad + 2[a_n(t) - a_{\infty}(t)] \langle (-\Delta)^{1\slash 2}u_t^{\infty}, v_t \rangle_X + 
2\int_{\Omega} [f(u^n) - f(u^{\infty})] u_t \, dx \\
&\leq 2\vert a_{\infty}(t) - a_n(t)\vert \Vert v_t^{\infty} \Vert_{X^{1\slash 2}} 
\left( \Vert u_t^n \Vert_X + \Vert u_t^{\infty} \Vert_X \right) \\ & \quad  
+ 2\vert a_n(t) - a_{\infty}(t)\vert \Vert u_t^{\infty} \Vert_{X^{1\slash 2}} 
\left( \Vert v_t^n \Vert_X + \Vert v_t^{\infty} \Vert_X \right) 
+ 2 \int_{\Omega} \vert (f(u^n) - f(u^{\infty})) u_t \vert \, dx.
\end{split}
\]

From the regularity theorem of the pullback attractor (see \cite[Theorem 2.3]{BNS}), there exists a constant $C_1>0$, which is independent of both the parameter $\epsilon$ and $t$, such that
\[
\Vert u_t^{\infty} \Vert_{X^{1\slash 2}}, \, 
\Vert v_t^{\infty} \Vert_{X^{1\slash 2}} \leq C_1.
\]

Moreover, due to the theorem on existence of global solution, it follows that the norms
\[
\Vert u_t^n \Vert_X, \, \Vert u_t^{\infty} \Vert_X, \, \Vert v_t^n \Vert_X, \, 
\Vert v_t^{\infty} \Vert_X
\]
are all bounded and this boundedness depends only on the initial data (see the proof of \cite[Theorem 2.1]{BNS}).
Thus, it turns out that we have
\begin{equation}\label{est stability 01}
\begin{split}
&\frac{d}{dt}\left( \Vert u\Vert_{X^{1\slash 2}}^2 + \Vert u\Vert_X^2 + \Vert u_t\Vert_X^2 + 
\Vert v\Vert_{X^{1\slash 2}}^2 + \Vert v_t\Vert_X^2 \right) \\
&\leq C^{\prime}\vert a_n(t) - a_{\infty}(t)\vert + 
2 \int_{\Omega} \vert (f(u^n) - f(u^{\infty})) u_t \vert \, dx,
\end{split}
\end{equation}
for some constant $C^{\prime}>0$.

Now, let us deal with the integral term $\int_{\Omega} \vert (f(u^n) - f(u^{\infty})) u_t \vert \, dx$. In fact, using the Mean Value Theorem, there exists $\sigma\in (0, 1)$ such that
\[
\vert f(u^n) - f(u^{\infty}) \vert = 
\vert f^{\prime} (\sigma u^n + (1 - \sigma)u^{\infty}) \vert \, \vert u^n - u^{\infty} \vert = 
\vert f^{\prime} (\sigma u^n + (1 - \sigma)u^{\infty}) \vert \, \vert u\vert.
\]

On the other hand, since $\frac{\rho - 1}{2\rho} + \frac{1}{2\rho} + \frac12 = 1$, then we apply H\"{o}lder inequality to obtain
\begin{equation}\label{est stability 02}
\begin{split}
&\int_{\Omega} \vert (f(u^n) - f(u^{\infty})) u_t \vert \, dx 
= \int_{\Omega} \vert f^{\prime} (\sigma u^n + (1 - \sigma)u^{\infty}) \vert \, \vert u\vert \, 
\vert u_t\vert \, dx \\
&\leq \Vert f^{\prime} (\sigma u^n + (1 - \sigma)u^{\infty}) \Vert_{L^{\frac{2\rho}{\rho - 1}}(\Omega)} 
\Vert u \Vert_{L^{2\rho}(\Omega)} 
\Vert u_t \Vert_{L^2(\Omega)}.
\end{split}
\end{equation}

The embedding $X^{1\slash 2} \hookrightarrow L^{2\rho}(\Omega)$ holds as $\rho \in \left[\frac{N-1}{N-2}, \frac{N}{N-2}\right)$. Thus, we can use the growth condition \eqref{Gcondition} to obtain a bound for the norm 
$
\Vert f^{\prime} (\sigma u^n + (1 - \sigma)u^{\infty}) \Vert_{L^{\frac{2\rho}{\rho - 1}}(\Omega)}
$
as follows
\begin{equation}\label{est stability 03}
\begin{split}
&\Vert f^{\prime} (\sigma u^n + (1 - \sigma)u^{\infty}) \Vert_{L^{\frac{2\rho}{\rho - 1}}(\Omega)} \leq 
\left( \int_{\Omega} 
[c(1 + \vert\sigma u^n + (1 - \sigma)u^{\infty} \vert^{\rho - 1}) ]^{\frac{2\rho}{\rho - 1}} \, dx 
\right)^{\frac{\rho - 1}{2\rho}} \\
&\leq c_1\left( \vert\Omega\vert^{\frac{\rho - 1}{2\rho}} + 
\Vert \sigma u^n + (1 - \sigma)u^{\infty} \Vert_{L^{2\rho}(\Omega)}^{\rho - 1} \right) \\
&\leq c_2 \left( 1 + \Vert u^n \Vert_{L^{2\rho}(\Omega)}^{\rho - 1} + 
\Vert u^{\infty} \Vert_{L^{2\rho}(\Omega)}^{\rho - 1} \right) \\
&\leq c_3 \left( 1 + \Vert u^n \Vert_{X^{1\slash 2}}^{\rho - 1} + 
\Vert u^{\infty} \Vert_{X^{1\slash 2}}^{\rho - 1} \right) 
\leq C_0,
\end{split}
\end{equation}
where the norms $\Vert u^n \Vert_{X^{1\slash 2}}$ and $\Vert u^{\infty} \Vert_{X^{1\slash 2}}$ are bounded by the theorem on existence of global solution, and this boundedness depends only on the initial data (see the proof of \cite[Theorem 2.1]{BNS}).

Thus, combining the estimates \eqref{est stability 01}, \eqref{est stability 02} and \eqref{est stability 03}, we obtain
\[
\begin{split}
&\frac{d}{dt}\left( \Vert u\Vert_{X^{1\slash 2}}^2 + \Vert u\Vert_X^2 + \Vert u_t\Vert_X^2 + 
\Vert v\Vert_{X^{1\slash 2}}^2 + \Vert v_t\Vert_X^2 \right) \\
&\leq C^{\prime}\vert a_n(t) - a_{\infty}(t)\vert + 
2C_0 \Vert u \Vert_{L^{2\rho}(\Omega)} 
\Vert u_t \Vert_{L^2(\Omega)} \\
&\leq C^{\prime}\vert a_n(t) - a_{\infty}(t)\vert + 2\hat{C} 
\Vert u\Vert_{X^{1\slash 2}} \Vert u_t \Vert_X \\
&\leq C^{\prime}\vert a_n(t) - a_{\infty}(t)\vert + \hat{C} \left( 
\Vert u\Vert_{X^{1\slash 2}}^2 + \Vert u\Vert_X^2 + \Vert u_t\Vert_X^2 + 
\Vert v\Vert_{X^{1\slash 2}}^2 + \Vert v_t\Vert_X^2 \right),
\end{split}
\]
that is,
\[
\frac{d}{dt} G(t) \leq C^{\prime\prime}\vert a_n(t) - a_{\infty}(t)\vert + 
C^{\prime\prime} G(t), \quad t > \tau,
\]
where $C^{\prime\prime} = \max\{ C^{\prime}, \hat{C} \} > 0$ and
$$
G(t) = \Vert u(t)\Vert_{X^{1\slash 2}}^2 + \Vert u(t)\Vert_X^2 + \Vert u_t(t)\Vert_X^2 + 
\Vert v(t)\Vert_{X^{1\slash 2}}^2 + \Vert v_t(t)\Vert_X^2.
$$

Since
\[
(u^n(\tau), u_t^n(\tau), v^n(\tau), v_t^n(\tau)) = (u_0, u_1, v_0, v_1) = 
(u^{\infty}(\tau), u_t^{\infty}(\tau), v^{\infty}(\tau), v_t^{\infty}(\tau)),
\]
$u = u^n - u^{\infty}$ and $v = v^n - v^{\infty}$, then $G(\tau) = 0$. With this, we can integrate the inequality
\[
\begin{split}
\frac{d}{ds} \left( G(s) e^{-C^{\prime\prime}(s-\tau)} \right) &\leq 
C^{\prime\prime}\vert a_n(s) - a_{\infty}(s)\vert e^{-C^{\prime\prime}(s-\tau)}, \quad \tau\leq s\leq t, \\
\end{split}
\]
from $\tau$ to $t$ in order to obtain
\[
G(t) e^{-C^{\prime\prime}(t-\tau)} 
\leq - \Vert a_n - a_{\infty} \Vert_{L^{\infty}([\tau, t])} e^{-C^{\prime\prime}(t-\tau)} 
+ \Vert a_n - a_{\infty} \Vert_{L^{\infty}([\tau, t])},
\]
that is,
\[
G(t) \leq - \Vert a_n - a_{\infty} \Vert_{L^{\infty}([\tau, t])} + 
\Vert a_n - a_{\infty} \Vert_{L^{\infty}([\tau, t])} e^{C^{\prime\prime}(t-\tau)}.
\]

Hence, as $a_n(s) \xrightarrow{n\rightarrow\infty} a_{\infty}(s)$ uniformly for $s$ in compact subsets of $\mathbb{R}$, it follows that
\[
\begin{split}
\Vert W^n(t) - W^{\infty}(t) \Vert_{Y_0}^2 
& = \Vert (u(t), u_t(t), v(t), v_t(t)) \Vert_{Y_0}^2 \\
&= \Vert u(t)\Vert_{X^{1\slash 2}}^2 + \Vert u_t(t)\Vert_X^2 + 
\Vert v(t)\Vert_{X^{1\slash 2}}^2 + \Vert v_t(t)\Vert_X^2 \\
&\leq G(t) \leq 
\Vert a_n - a_{\infty} \Vert_{L^{\infty}([\tau, t])} e^{C^{\prime\prime}(t-\tau)} 
\xrightarrow{n\rightarrow\infty} 0,
\end{split}
\]
in compact subsets of $\mathbb{R}$.
\end{proof}

 Next, we aim to prove that  if $\xi\colon\mathbb{R}\to Y_0$ is a bounded global solution of the problem \eqref{limit system} and $\xi^n\colon\mathbb{R}\to Y_0$ is a bounded global solution of the problem \eqref{problem n}, $n \in \N$, then $\xi(t_n + t) = \xi^n(t)$ for all $t \in \R$. For that, we exhibit some auxiliary results.

According to \cite[Proposition 4.1]{BNS}, for each fixed $t \in \R$ and $\epsilon \in [0,1]$, the operator $\mathcal{A}_{\epsilon}(t)$ defined in \eqref{abstract operator}-\eqref{domain abstract operator} is maximal accretive (i.e., $-\mathcal{A}_{\epsilon}(t)$ is dissipative and $R(I+ \mathcal{A}_{\epsilon}(t)) = Y_0$). 
By the Lumer-Phillips Theorem (\cite[Theorem 4.3]{Pazy}), we have $-\mathcal{A}_{\epsilon}(t)$ is the infinitesimal generator of a $C_0$-semigroup of contractions in $Y_0$ given by $\{e^{-\tau \mathcal{A}_{\epsilon}(t)}: \tau \geq 0\}$, for each $t\in \R$ and $\epsilon \in [0,1]$. Consequently, for each $t\in \R$ and $\epsilon \in [0,1]$, it follows by \cite[Corollary 3.6]{Pazy} that  $\rho(-\mathcal{A}_{\epsilon}(t)) \supset \{\lambda \in \C: {\rm Re}(\lambda) > 0\}$ and
\begin{equation}\label{estimate resolt}
\|(\lambda I + \mathcal{A}_{\epsilon}(t))^{-1}\|_{\mathcal{L}(Y_0)} \leq \dfrac{1}{{\rm Re}(\lambda)},
\end{equation}
whenever $\lambda \in \C$ with ${\rm Re}(\lambda) > 0$.

On the other hand, as proved in \cite[Theorem 4.7]{BNS}, for each fixed $t\in\mathbb{R}$ and $\epsilon \in [0,1]$, the operator $\mathcal{A}_{\epsilon}(t)$ satisfies
\begin{equation}\label{estimate resolt imaginarium}
\vert\beta\vert \Vert (i\beta I + \mathcal{A}_{\epsilon}(t))^{-1} \Vert_{\mathcal{L}(Y_0)} \leq M,
\end{equation}
for all $\beta\in\mathbb{R}$, where $M = 1 + 2\left(2\eta + \frac{2(a_1^2 + 1)}{\eta} + 4\right) + \frac{2a_1 + 2}{\eta} > 0$ is independent of $t\in\mathbb{R}$ and $\epsilon\in [0, 1]$. 


\begin{lemma}\label{Lemma 3.2 aux AL1} There exists a constant $\mathcal{C}>0$, which  is independent of $t\in\mathbb{R}$ and $\epsilon\in [0, 1]$, such that
$\|(\lambda I + \mathcal{A}_{\epsilon}(t))^{-1}\|_{\mathcal{L}(Y_0)} \leq \dfrac{\mathcal{C}}{1 +|\lambda|}$ uniformly in $t\in\mathbb{R}$ and $\epsilon\in [0, 1],$ for all $\lambda$ in a sector $\Gamma$ whose angle is greater than $\frac{\pi}{2}$.
\end{lemma}

\begin{proof}  Let $t\in\mathbb{R}$ and $\epsilon\in [0, 1]$. By the proof of \cite[Proposition 4.3]{BNS}, there exists a constant $N >0$, independent of $t$ and $\epsilon$, such that $\|\mathcal{A}_{\epsilon}^{-1}(t)\|_{\mathcal{L}(Y_0)} \leq N$. Thus,  $\|\lambda \mathcal{A}_{\epsilon}^{-1}(t)\|_{\mathcal{L}(Y_0)} \leq \frac{1}{2}$ provided $|\lambda| \leq \frac{1}{2N}$, $\lambda \in \C$. Therefore, $(\lambda I +  \mathcal{A}_{\epsilon}(t))^{-1}$ exists for $|\lambda| \leq \frac{1}{2N}$ and
\[
\|(\lambda I +  \mathcal{A}_{\epsilon}(t))^{-1}\|_{\mathcal{L}(Y_0)} \leq \|(I + \lambda \mathcal{A}_{\epsilon}^{-1}(t))^{-1}\|_{\mathcal{L}(Y_0)}\|\mathcal{A}_{\epsilon}^{-1}(t)\|_{\mathcal{L}(Y_0)} \leq 2N.
\]
Consequently,
\begin{equation}\label{limit on a bounded p1}
\|(\lambda I + \mathcal{A}_{\epsilon}(t))^{-1}\|_{\mathcal{L}(Y_0)} \leq \frac{1}{|\lambda|},
\end{equation}
whenever $|\lambda|\leq \frac{1}{2N}$. 

Now, let $\varphi \in(0, \frac{\pi}{2})$ be such that $\tan(\varphi) = 4M$, where $M>0$ comes from \eqref{estimate resolt imaginarium}. Set $\Gamma_1:= \{\lambda \in \C \colon |\lambda| \geq \frac{1}{2N} \; \text{and} \; |arg(\lambda)| \leq \varphi\}$.

\vspace{.2cm}

\textbf{Claim 1:} $\|(\lambda I + \mathcal{A}_{\epsilon}(t))^{-1}\|_{\mathcal{L}(Y_0)}  \leq \dfrac{\frac{2}{1-\sin(\varphi)}}{|\lambda|}$ for all $\lambda \in \Gamma_1$.

\vspace{.2cm}

Indeed, let $\lambda \in \Gamma_1$. Then
\[
|\lambda| \leq |{\rm Re}(\lambda)| + |{\rm Im}(\lambda)| \leq  |{\rm Re}(\lambda)| + \sin(\varphi)|\lambda|.
\]
Consequently, it follows by \eqref{estimate resolt} that
\[
\|(\lambda I + \mathcal{A}_{\epsilon}(t))^{-1}\|_{\mathcal{L}(Y_0)} \leq \frac{1}{{\rm Re}(\lambda)} \leq \frac{\frac{2}{1-\sin(\varphi)}}{|\lambda|}.
\]

\vspace{.2cm} 

Define $\Gamma_2 =  \bigcup_{t\neq 0}\{\mu \in \C \colon |\mu - it| < \frac{|t|}{2M}\}$.

\textbf{Claim 2:}   If $\mu\in \Gamma_2$, then $\Vert (\mu I + \mathcal{A}_{\epsilon}(t))^{-1} \Vert_{\mathcal{L}(Y_0)} \leq 
\dfrac{2M+1}{\vert\mu\vert}$.

\vspace{.2cm} 

In fact, let $\mu\in\Gamma_2$ and $t \in \R$, $t \neq 0$, be such that $\vert\mu - it\vert < \frac{\vert t\vert}{2M}$.  Using \eqref{estimate resolt imaginarium}, we deduce
\[
\begin{split}
\Vert (\mu - it) (itI + \mathcal{A}_{\epsilon}(t))^{-1} \Vert_{\mathcal{L}(Y_0)} 
&= \vert\mu - it\vert \Vert (itI + \mathcal{A}_{\epsilon}(t))^{-1} \Vert_{\mathcal{L}(Y_0)} < \frac12
\end{split}
\]
and, thanks to the Neumann series, we obtain
\[
\Vert [ (\mu - it) (itI + \mathcal{A}_{\epsilon}(t))^{-1} + I ]^{-1} \Vert_{\mathcal{L}(Y_0)} 
\leq 2.
\]
Now, using \eqref{estimate resolt imaginarium} again, we get
\[
\begin{split}
\Vert (\mu I + \mathcal{A}_{\epsilon}(t))^{-1} \Vert_{\mathcal{L}(Y_0)} &= \Vert [ (\mu - it) (itI + \mathcal{A}_{\epsilon}(t))^{-1} + I ]^{-1} 
(itI + \mathcal{A}_{\epsilon}(t))^{-1} \Vert_{\mathcal{L}(Y_0)} \\
&\leq \frac{2M}{\vert t\vert} \leq \left( \frac{\vert\mu - it\vert}{\vert t\vert} + 1\right) \frac{2M}{\vert\mu\vert} <  \frac{2M+1}{\vert\mu\vert}.
\end{split}
\]

Since $\partial\{\lambda \in \Gamma_1: |arg(\lambda)| = \varphi\} = \left\{\left(\frac{|t|}{4M}, t \right): t \in \R\right\} \subset \Gamma_2$, it follows by \eqref{limit on a bounded p1}, Claim 1 and Claim 2 that we have construct a sector $\Gamma$ whose angle is greater than $\frac{\pi}{2}$ such that
\begin{equation}\label{teo new part1}
\Vert (\lambda I + \mathcal{A}_{\epsilon}(t))^{-1} \Vert_{\mathcal{L}(Y_0)} \leq 
\frac{\tilde{C}}{\vert\lambda\vert},
\end{equation}
where $\tilde{C} = \max\{2M+1, \frac{2}{1-\sin(\varphi)}\} >0$ is independent of $t\in\mathbb{R}$ and $\epsilon\in [0, 1]$.

\vspace{.2cm}

Lastly, for a given $\lambda \in \Gamma$, we have
\begin{equation}\label{lim A(I)}
\|\mathcal{A}_{\epsilon}(t)(\lambda I + \mathcal{A}_{\epsilon}(t))^{-1}\|_{\mathcal{L}(Y_0)}  \leq 1 + |\lambda|\Vert (\lambda I + \mathcal{A}_{\epsilon}(t))^{-1} \Vert_{\mathcal{L}(Y_0)}  \leq 1 + \tilde{C},
\end{equation}
which together with the uniform boundedness of  $\|\mathcal{A}_{\epsilon}^{-1}(t)\|_{\mathcal{L}(Y_0)}$, we conclude that
\begin{equation}\label{teo new part2}
 \Vert (\lambda I + \mathcal{A}_{\epsilon}(t))^{-1} \Vert_{\mathcal{L}(Y_0)}  \leq \| \mathcal{A}_{\epsilon}(t)(\lambda I + \mathcal{A}_{\epsilon}(t))^{-1}\|_{\mathcal{L}(Y_0)}  \| \mathcal{A}_{\epsilon}^{-1}(t)\|_{\mathcal{L}(Y_0)} \leq K,
\end{equation}
for some constant $K>0$ independent of $t\in\mathbb{R}$ and $\epsilon\in [0, 1]$.

From \eqref{teo new part1} and  \eqref{teo new part2}, 
\[
(1+ |\lambda|)\|(\lambda I + \mathcal{A}_{\epsilon}(t))^{-1}\|_{\mathcal{L}(Y_0)}  \leq \tilde{C} + K := \mathcal{C},
\]
for all $\lambda \in \Gamma$. 
\end{proof}

 \begin{remark}\label{R1 first} \rm Since the estimates \eqref{lim A(I)} and \eqref{teo new part2} do not depend on $\epsilon$, then there exist $\delta > 0$  and $K > 0$ such that:
 \begin{enumerate}
 \item[$(i)$] $\|e^{-\tau\mathcal{A}_{\epsilon}(t)}\|_{\mathcal{L}(Y_0)} \leq Ke^{-\delta\tau}$, for all $\tau \geq 0$ and $t \in \R$;
 \item[$(ii)$] $\|\mathcal{A}_{\epsilon}(t)e^{-\tau\mathcal{A}_{\epsilon}(t)}\|_{\mathcal{L}(Y_0)} \leq \dfrac{Ke^{-\delta\tau}}{\tau}$, for all $\tau>0$ and $t \in \R$.
 \end{enumerate}
 \end{remark}

 \begin{lemma}\label{Lemma 3.3 aux AL2}  For any $T>0$, there are constants $C > 0$ and $0 < \beta \leq 1$, both independent of $T$ and $\epsilon$, such that
\[
\Vert [ \mathcal{A}_{\epsilon}(t) - \mathcal{A}_{\epsilon}(s) ] 
\mathcal{A}_{\epsilon}^{-1}(\tau) \Vert_{\mathcal{L}(Y_0)} \leq 
C\vert t-s\vert^{\beta}
\]
for every $t, s, \tau \in [-T, T]$. In particular, $\Vert \mathcal{A}_{\epsilon}(t) 
\mathcal{A}_{\epsilon}^{-1}(\tau) \Vert_{\mathcal{L}(Y_0)} \leq 
1 + 2a_1$, for all $t, \tau \in \R$, where the constant $a_1$ comes from \eqref{function a is bounded}.
\end{lemma}

\begin{proof}
Let $T>0$ be fixed and arbitrary, and let $t, s, \tau \in [-T, T]$. 
Since 
\[
 \mathcal{A}_{\epsilon}(t) - \mathcal{A}_{\epsilon}(s) \, \!\!=\!\! \,
[a_{\epsilon}(t) - a_{\epsilon}(s)]
\begin{bmatrix}
0 & 0 & 0 & 0 \\
0 & 0 & 0 & A^{1\slash 2} \\
0 & 0 & 0 & 0 \\
0 & -A^{1\slash 2} & 0 & 0
\end{bmatrix} 
\]
and
\[
\small
\mathcal{A}_{\epsilon}^{-1}(\tau) =
\begin{bmatrix}
\eta A^{1\slash 2} (A+I)^{-1} & (A+I)^{-1} & a_{\epsilon}(\tau)A^{1\slash 2} (A+I)^{-1} & 0 \\
-I                                           & 0                & 0                                          & 0 \\
- a_{\epsilon}(\tau) A^{-1\slash 2}                 & 0               & \eta A^{-1\slash 2}            & A^{-1} \\
0                                            & 0               & -I                                         & 0
\end{bmatrix}, 
\]
we obtain
\[
\begin{split}
\left\Vert [\mathcal{A}_{\epsilon}(t) - \mathcal{A}_{\epsilon}(s) ] 
\mathcal{A}_{\epsilon}^{-1}(\tau) 
\begin{bmatrix}
u\\ v\\ w\\ z
\end{bmatrix}
\right\Vert_{Y_0} 
&
= \vert a_{\epsilon}(t) - a_{\epsilon}(s) \vert 
\left( \left\Vert -A^{1\slash 2}w \right\Vert_X 
+ \left\Vert A^{1\slash 2}u \right\Vert_X \right) \\
&\leq C\vert t-s\vert^{\beta} 
\left( \left\Vert u\right\Vert_{X^{1\slash 2}} + \left\Vert v\right\Vert_X 
+ \left\Vert w\right\Vert_{X^{1\slash 2}} 
+ \left\Vert z\right\Vert_X \right) \\
&= C\vert t-s\vert^{\beta} 
\left\Vert
\begin{bmatrix}
u\\ v\\ w\\ z
\end{bmatrix} 
\right\Vert_{Y_0}
\end{split}
\]
for every $[u\,\, v\,\, w\,\, z]^T \in Y_0$,
where we have used \eqref{hol-a}. 
Hence,
\[
\Vert [\mathcal{A}_{\epsilon}(t) - \mathcal{A}_{\epsilon}(s)] 
 \mathcal{A}_{\epsilon}^{-1}(\tau) 
\Vert_{\mathcal{L}(Y_0)} \leq C\vert t-s\vert^{\beta}
\]
for every $t, s, \tau \in [-T, T],$ and the result is proved.
\end{proof}

\begin{proposition}\label{Prop 3.5 rel} Let $\xi\colon\mathbb{R}\to Y_0$ be a bounded global solution of the problem \eqref{limit system} and $\xi^n\colon\mathbb{R}\to Y_0$ be a bounded global solution of the problem \eqref{problem n}, $n \in \N$. Then $\xi(t_n + t) = \xi^n(t)$ for all $t \in \R$.
\end{proposition}
\begin{proof} Consider the homogeneous problems
\begin{equation}\label{limit system homo a_0}
\begin{cases}
u_{tt} - \Delta u + u + \eta(-\Delta)^{1\slash 2}u_t + a_0(t)(-\Delta)^{1\slash 2}v_t = 0, & (x, t) \in\Omega\times (\tau, \infty), \\
v_{tt} - \Delta v + \eta(-\Delta)^{1\slash 2}v_t - a_0(t)(-\Delta)^{1\slash 2}u_t = 0, & (x, t) \in\Omega\times (\tau, \infty), \\
u = v = 0, & (x, t) \in\partial\Omega\times (\tau, \infty), \\
(u(\tau), u_t(\tau), v(\tau), v_t(\tau)) = (u_0, u_1, v_0, v_1),
\end{cases}
\end{equation}
and
\begin{equation}\label{limit system homo a_n}
\begin{cases}
u_{tt} - \Delta u + u + \eta(-\Delta)^{1\slash 2}u_t + a_n(t)(-\Delta)^{1\slash 2}v_t = 0, & (x, t) \in\Omega\times (\tau, \infty), \\
v_{tt} - \Delta v + \eta(-\Delta)^{1\slash 2}v_t - a_n(t)(-\Delta)^{1\slash 2}u_t = 0, & (x, t) \in\Omega\times (\tau, \infty), \\
u = v = 0, & (x, t) \in\partial\Omega\times (\tau, \infty), \\
(u(\tau), u_t(\tau), v(\tau), v_t(\tau)) = (u_0, u_1, v_0, v_1) \in Y_0,
\end{cases}
\end{equation}

Let $\{L(t,s) \colon t \geq s \in \R\}$ and $\{L_n(t,s) \colon t \geq s \in \R\}$ be the linear evolution processes in $Y_0$ associated with  \eqref{limit system homo a_0} and \eqref{limit system homo a_n}, respectively. Since Lemmas \ref{Lemma 3.2 aux AL1}  and \ref{Lemma 3.3 aux AL2} hold, it follows by \cite{Sobo} that
\[
L(t,s)  =  e^{-(t-s)\mathcal{A}_0(s)} + \int_{s}^{t} L(t, \tau)[\mathcal{A}_0(s) - \mathcal{A}_0(\tau)]e^{-(\tau-s)\mathcal{A}_0(s)}d\tau,
\]
and
\[
L_n(t,s) = e^{-(t-s)\mathcal{A}_n(s)} + \int_{s}^{t} L_n(t, \tau)[\mathcal{A}_n(s) - \mathcal{A}_n(\tau)]e^{-(\tau-s)\mathcal{A}_n(s)}d\tau,
\]
where  \begin{equation*}\label{abstract operator with n}
\begin{split}
\mathcal{A}_{n}(t) &=\!\!
\begin{bmatrix}
0     & -I                           & 0 & 0 \\
A+I & \eta A^{1\slash 2} & 0 & a_{n}(t) A^{1\slash 2} \\
0     & 0                            & 0 & -I \\
0     & - a_{n}(t) A^{1\slash 2} & A & \eta A^{1\slash 2}
\end{bmatrix},
\!\! \\
\end{split}
\end{equation*}
with $t \in \R$ and $n \in \N\cup\{0\}$. Moreover,
\[
\xi(t) = L(t, s)\xi(s) +  \int_{s}^t L(t,\tau)F(\xi(\tau))d\tau, \quad t, s  \in \R,
\]
and
\[
\xi^n(t) =  L_n(t, s)\xi(s) +  \int_{s}^t L_n(t,\tau)F(\xi^n(\tau))d\tau, \quad t, s \in \R, \; n \in \N.
\]
According to the proof of  \cite[Theorem 2.3]{BNS}, we conclude that
\[
\xi(t) = \int_{-\infty}^t L(t,s)F(\xi(s))ds, \quad t \in \R,
\]
and
\[
\xi^n(t) = \int_{-\infty}^t L_n(t,s)F(\xi^n(s))ds, \quad t \in \R, \; n \in \N,
\]
with $F$ is given by \eqref{nonlinearity F}. Furthermore, note that $\mathcal{A}_0(s+t_n)  = \mathcal{A}_n(s)$ for all $s \in \R$ and $n \in \N$. Consequently, we deduce
\[
\begin{split} 
 & L(t+t_n,s+t_n) =\\ 
 & = e^{-(t-s)\mathcal{A}_0(s+t_n)} + \int_{s+t_n}^{t+t_n} L(t+t_n, \tau)[\mathcal{A}_0(s+t_n) - \mathcal{A}_0(\tau)]e^{-(\tau-s-t_n)\mathcal{A}_0(s+t_n)}d\tau\\
 & = e^{-(t-s)\mathcal{A}_n(s)} + \int_{s}^{t} L(t+t_n, \tau+t_n)[\mathcal{A}_0(s+t_n) - \mathcal{A}_0(\tau+t_n)]e^{-(\tau-s)\mathcal{A}_0(s+t_n)}d\tau\\
 & = e^{-(t-s)\mathcal{A}_n(s)} + \int_{s}^{t} L(t+t_n, \tau+t_n)[\mathcal{A}_n(s) - \mathcal{A}_n(\tau)]e^{-(\tau-s)\mathcal{A}_n(s)}d\tau,
\end{split} 
\]
for all $t, s \in \R$ and $n\in \N$. Using Remark \ref{R1 first}, Lemma \ref{Lemma 3.3 aux AL2} and the Gronwall inequality, we conclude that $L(t+t_n,s+t_n)  = L_n(t, s)$ for all $t \geq s \in \R$.  In conclusion,
\[
\xi(t+t_n) = \int_{-\infty}^{t+t_n} L(t+t_n,s)F(\xi(s))ds = \int_{-\infty}^{t} L(t+t_n,s+t_n)F(\xi(s+t_n))ds 
\]
\[
 =\int_{-\infty}^{t} L_n(t,s)F(\xi(s+t_n))ds,
\]
for all $t \in \R$ and $n\in \N$. Since (see \cite[Proposition 4.9]{BNS})
\[
\|F(z_1) - F(z_2)\|_{Y_{\alpha -1}} \leq c\|z_1-z_2\|_{Y_0}(1 + \|z_1\|_{Y_0}^{\rho -1} + \|z_2\|_{Y_0}^{\rho -1}), \quad z_1, z_2 \in Y_0, \quad \alpha \in (0, 1),
\]
where $Y_{\alpha -1} = [Y_{-1}, Y_{0}]_{\alpha}$ (complex interpolation), and $\|L(t,s)\|_{\mathcal{L}(Y_{\alpha-1}, Y_{0})} \leq c(\alpha)(t-s)^{\alpha-1}$ (by choosing the phase space $Z_0 = Y_{-1}$, $Z_{\alpha} = Y_{\alpha -1}$ and the closed extension of $\mathcal{A}_{\epsilon}(t)$ to $Y_0$, and take into account that \cite[Proposition 4.4.]{BNS} holds, we have $\|L(t,s)\|_{\mathcal{L}(Y_{\alpha-1}, Y_{0})} = \|L(t,s)\|_{\mathcal{L}(Z_{\alpha}, Z_{1})}$), it follows again by using the  Gronwall inequality,
that $\xi(t_n + t) = \xi^n(t)$ for all $t \in \R$.
\end{proof}

\begin{corollary}
If $\xi\colon\mathbb{R}\to Y_0$ and  $\xi^{\infty}\colon\mathbb{R}\to Y_0$ are bounded global solutions associated to the problems \eqref{problem n} and \eqref{problem with limit function}, respectively, starting in their pullback attractors, then
\[
\Vert \xi(t+t_n) - \xi^{\infty}(t) \Vert_{Y_0}^2 \xrightarrow{n\rightarrow\infty} 0
\]
in compact subsets of $\mathbb{R}$.
\end{corollary}

\subsection{Gradient-like structure of the limit pullback attractor}

From here on,  we shall assume that the problem \eqref{edp01}-\eqref{cond01} admits only a finite number of equilibrium points (time independent solutions), that is, there exist only finitely many solutions $\{(u_1^{\ast}, 0), \ldots, (u_r^{\ast}, 0)\}$ of the elliptic problem
\[
\begin{cases}
(- \Delta)u + u = f(u) & \mbox{in} \quad \Omega, \\
(- \Delta)v = 0 & \mbox{in} \quad \Omega, \\
u = v = 0 & \mbox{on} \quad \partial\Omega,
\end{cases}
\]
and we will denote the set of equilibria as $\mathcal{E}^{\ast} = \{e_1^{\ast}, \ldots, e_r^{\ast} \}$, where each $e_j^{\ast}$ is of the form
\[
e_j^{\ast} \, \!\!=\!\!
\begin{bmatrix}
u_j^{\ast}\\ 0\\ 0\\ 0
\end{bmatrix},
\quad 1\leq j\leq r.
\]
Further, we shall assume that all equilibrium points in $\mathcal{E}^{\ast}$ are hyperbolic, in the sense of Definition \ref{def hyperbolic sol}, 
for the limit problem ($\epsilon = 0$) associated to the system \eqref{edp01}-\eqref{cond01}. 

In order to pursue a result on lower semicontinuity of pullback attractors at $\epsilon = 0$, we will assume the following hyperbolicity condition:
\begin{equation}\label{hyper condition}
0 \notin \sigma (A+I - f^{\prime}(u)) \quad \mbox{whenever} \quad 
u\in H_0^1(\Omega)
\end{equation}
and the spectrum $\sigma (A+I - f^{\prime}(u))$ has no intersection with the imaginary axis.


\begin{proposition}\label{prop de Lyapunov}
The evolution system \eqref{edp01}-\eqref{cond01} admits a Lyapunov functional.
\end{proposition}

\begin{proof}
Let $(u(t), u_t(t), v(t), v_t(t))$ be the global solution of the problem \eqref{edp01}-\eqref{cond01} in $Y_0$. Multiplying the first equation in \eqref{edp01} by $u_t$, and the second by $v_t$, we obtain
\[
\begin{split}
&\frac{1}{2}\frac{d}{dt}\int_{\Omega}\vert u_t\vert^2 \, dx 
+ \frac{1}{2}\frac{d}{dt}\int_{\Omega}\vert\nabla u\vert^2 \, dx 
+ \frac{1}{2}\frac{d}{dt}\int_{\Omega}\vert u\vert^2 \, dx 
+ \eta\Vert (-\Delta)^{1\slash 4} u_t \Vert_X^2 \\
&+ a_{\epsilon}(t) \langle (-\Delta)^{1\slash 2} v_t, u_t \rangle_X = 
\frac{d}{dt} \int_{\Omega} \int_0^u f(s) \, dsdx
\end{split}
\]
and
\[
\frac{1}{2}\frac{d}{dt}\int_{\Omega}\vert v_t\vert^2 \, dx 
+ \frac{1}{2}\frac{d}{dt}\int_{\Omega}\vert\nabla v\vert^2 \, dx 
+ \eta\Vert (-\Delta)^{1\slash 4} v_t \Vert_X^2 
- a_{\epsilon}(t) \langle (-\Delta)^{1\slash 2} u_t, v_t \rangle_X = 0,
\]
for all $t > \tau$, which combined leads to
\begin{equation}\label{derivative func 01}
\begin{split}
&\frac{d}{dt} \frac12 \left( \Vert u\Vert_{X^{1\slash 2}}^2 + \Vert u\Vert_X^2 + 
\Vert u_t\Vert_X^2 + \Vert v\Vert_{X^{1\slash 2}}^2 + \Vert v_t\Vert_X^2 - 
2\int_{\Omega} \int_0^u f(s) \, dsdx \right) \\
&= - \eta\Vert (-\Delta)^{1\slash 4} u_t \Vert_X^2  - 
\eta\Vert (-\Delta)^{1\slash 4} v_t \Vert_X^2
\end{split}
\end{equation}
for all $t > \tau$.

Thus, we may define an energy functional in the phase space $Y_0$ through the map $\mathcal{E} \colon Y_0 \to\mathbb{R}$ given by
\begin{equation}\label{energy func}
\begin{split}
\mathcal{E} (\phi, \varphi, \psi, \Phi) &= 
\frac12 \left( \Vert \phi \Vert_{X^{1\slash 2}}^2 + \Vert \phi \Vert_X^2 + 
\Vert \varphi \Vert_X^2 + \Vert \psi \Vert_{X^{1\slash 2}}^2 + 
\Vert \Phi \Vert_X^2 \right) \\
&- \int_{\Omega} \int_0^{\phi} f(s) \, dsdx,
\end{split}
\end{equation}
for $(\phi, \varphi, \psi, \Phi) \in Y_0$. By \eqref{derivative func 01} and \eqref{energy func}, it follows that
\[
\frac{d}{dt} \mathcal{E} (u(t), u_t(t), v(t), v_t(t)) 
= - \eta\Vert (-\Delta)^{1\slash 4} u_t \Vert_X^2  - 
\eta\Vert (-\Delta)^{1\slash 4} v_t \Vert_X^2
\]
for all $t > \tau$ and, therefore, the energy functional $\mathcal{E}$ is monotone decreasing along solutions.

Moreover, if $\mathcal{E}$ is a constant function with respect to $t\in\mathbb{R}$, then it follows from \eqref{derivative func 01} that $- \eta\Vert (-\Delta)^{1\slash 4} u_t \Vert_X^2  - 
\eta\Vert (-\Delta)^{1\slash 4} v_t \Vert_X^2 = 0$. Consequently, $\vert (-\Delta)^{1\slash 4}u_t \vert^2 =
\vert (-\Delta)^{1\slash 4}v_t \vert^2 = 0 \mbox{ almost everywhere on } 
\Omega$, i.e., $u_t = 0 \mbox{ and } v_t = 0 \mbox{ almost everywhere on } \Omega$.
Consequently, the global solutions where the energy functional $\mathcal{E}$ is constant in time $t\in\mathbb{R}$ must belong to the set of equilibria $\mathcal{E}^{\ast}$ (set of time independent solutions). This proves that $\mathcal{E}\colon Y_0\to\mathbb{R}$, defined in \eqref{energy func} is a Lyapunov functional for the evolution system \eqref{edp01}-\eqref{cond01}.
\end{proof}



In the next result, we consider the unstable manifold around an equilibria $e_i^{\ast}$, $1\leq i\leq r$, given by
\[
W^u(e_i^{\ast}) =  \{(t, \zeta) \in \R\times Y_0: \text{ there is a global solution } \xi\colon \R \to X  \text{ such that }
\]
\[
 \xi(t) = \zeta  \text{ and } \displaystyle\lim_{\tau \to -\infty} dist(\xi(\tau), e_i^{\ast}) =0\}
\]
and its section at time $t \in \R$
\[
W^u(e_i^{\ast})(t) = \{\zeta\in Y_0\colon (t, \zeta)\in W^u(e_i^{\ast}) \}.
\] 

\begin{theorem}\label{structure limit PA}
The limit pullback attractor $\{\mathbb{A}_0(t)\colon t\in\mathbb{R}\}$, associated to the problem \eqref{edp01}-\eqref{cond01} with $\epsilon = 0$, can be written as
\[
\mathbb{A}_0(t) = \bigcup_{i = 1}^r W^u(e_i^{\ast})(t) \quad \mbox{for all} \quad 
t\in\mathbb{R}.
\]
\end{theorem}

\begin{proof}
At first place, we will prove that any bounded global solution in the pullback attractor $\{\mathbb{A}_0(t)\colon t\in\mathbb{R}\}$ is both forwards and backwards asymptotic to some equilibria in $\mathcal{E}^{\ast}$. Indeed, let $\xi\colon\mathbb{R}\to Y_0$ be a bounded global solution of the problem \eqref{limit system} such that $\xi(t)\in \mathbb{A}_0(t)$ for all $t\in\mathbb{R}$.  Since $\bigcup_{t \in \R} \mathbb{A}_0(t)$ is compact (see \eqref{compactness limit set}) and the energy functional $\mathcal{E}$ is monotone decreasing along solutions, there exist
 $\ell_i, \ell_j \in\mathbb{R}$ such that
\begin{equation}\label{convergence of energy t}
\ell_i \xleftarrow{t\rightarrow -\infty} \mathcal{E}(\xi(t + s)) \xrightarrow{t\rightarrow +\infty} \ell_j
\end{equation}
for all $s \in \R$.

Let  $\{t_n\}_{n\in\mathbb{N}} \subset\mathbb{R}$ be a sequence such that $t_n \xrightarrow{n\rightarrow\infty} +\infty$. We may assume without loss of generality that
\[
a_0(t_n + s) = a_n(s) \xrightarrow{n\rightarrow\infty} a_{\infty}(s)
\]
uniformly for $s$ in compact subsets of $\mathbb{R}$. Consider problem \eqref{problem n} with this function $a_n$. By Theorem \ref{stability} and Proposition \ref{Prop 3.5 rel}, if
\[
\xi^n(\cdot) = (u^n(\cdot), u_t^n(\cdot), v^n(\cdot), v_t^n(\cdot)) 
\quad \mbox{and} \quad 
\xi^{\infty}(\cdot) = (u^{\infty}(\cdot), u_t^{\infty}(\cdot), v^{\infty}(\cdot), v_t^{\infty}(\cdot))
\]
are global solutions associated with the problems \eqref{problem n} and \eqref{problem with limit function}, respectively, then we have
\begin{equation}\label{convergence of xi(t_n + t)}
\xi(t_n + t) = \xi^n(t) \xrightarrow{n\rightarrow\infty} \xi^{\infty}(t) 
\quad \mbox{in} \quad Y_0
\end{equation}
in compact subsets of $\mathbb{R}$.  By \eqref{convergence of energy t} and \eqref{convergence of xi(t_n + t)},
\[
\mathcal{E}(\xi^{\infty}(t)) = \ell_j
\]
for all $t \in \R$.  Since the global solutions where the energy functional $\mathcal{E}$ is constant in $t\in\mathbb{R}$ are the equilibria in $\mathcal{E}^{\ast}$, there exists an equilibrium point $e_k^{\ast} \in \mathcal{E}^{\ast}$, for some $1\leq k\leq r$, such that $\xi^{\infty}(\cdot) \equiv e_k^{\ast}$. Hence,
\[
\xi(t_n + t) \xrightarrow{n\rightarrow +\infty} e_k^{\ast} \in \mathcal{E}^{\ast} 
\quad \mbox{in} \quad Y_0
\]
uniformly on compact subsets of $\mathbb{R}$. Similarly, we can take a sequence 
$\{t_n\}_{n\in\mathbb{N}} \subset\mathbb{R}$ with $t_n \xrightarrow{n\rightarrow\infty} -\infty$ and prove that  there exists an equilibrium point $e_p^{\ast} \in \mathcal{E}^{\ast}$, for some $1\leq p\leq r$, such that

\[
\xi(t_n + t) \xrightarrow{n\rightarrow \infty} e_p^{\ast} \in \mathcal{E}^{\ast} 
\quad \mbox{in} \quad Y_0
\]
uniformly on compact subsets of $\mathbb{R}$.

Now, we argue that this assertion does not depend on the particular choices of subsequences, and then we can conclude our result. Indeed, suppose that there are sequences $t_n, s_n \xrightarrow{n\rightarrow\infty} \infty$ such that
\[
\xi(t_n) \xrightarrow{n\rightarrow\infty} e_i^{\ast} \neq e_j^{\ast} 
\xleftarrow{n\rightarrow\infty} \xi(s_n).
\]
Reindexing these sequences, if necessary, we can assume without loss of generality that $t_{n+1} > s_n > t_n$ for all $n\in\mathbb{N}$.   Let $\delta = \|e_i^{\ast}- e_j^{\ast}\|_{Y_0} > 0$. There is $n_0 \in \N$ such that $\xi(t_n) \in B(e_i^{\ast}, \frac{\delta}{3})$ and $\xi(s_n) \in B(e_j^{\ast}, \frac{\delta}{3})$ for all $n \geq n_0$. Let $\tau_n\in (t_n, s_n)$ be such that $\xi(\tau_n) \notin B(e_j^{\ast}, \frac{\delta}{3}) \cup B(e_i^{\ast}, \frac{\delta}{3})$, for $n \geq n_0$. Thus, $\tau_n \xrightarrow{n\rightarrow\infty} \infty$ and, taking subsequences if necessary, we also have
\[
a_0(\tau_n + s) = a_n(s) \xrightarrow{n\rightarrow\infty} \tilde{a}(s)
\]
uniformly for $s$ in compact subsets of $\mathbb{R}$. Repeating the previous argument, we obtain
\[
\xi(\tau_n + t)  \xrightarrow{n\rightarrow\infty} \tilde{\xi}(t) 
\]
in compact subsets of $\mathbb{R}$, where $\tilde{\xi}$ is the global solution of the problem \eqref{problem with limit function} with $\bar{a}$ replaced by $\tilde{a}$, with 
\[
\tilde{\xi}(\cdot) \equiv e_m^{\ast} \in \mathcal{E}^{\ast}. 
\]
Since $\xi(\tau_n) \notin B(e_j^{\ast}, \frac{\delta}{3}) \cup B(e_i^{\ast}, \frac{\delta}{3})$, for $n \geq n_0$, we conclude that 
\[
 e_m^{\ast} \in \mathcal{E}^{\ast} \backslash 
\{e_i^{\ast}, e_j^{\ast}\}.
\]
Continuing with the previous argument, we obtain a contradiction due to the fact that there are only finitely many equilibria in $\mathcal{E}^{\ast}$. Hence, we can write the pullback attractor $\{\mathbb{A}_0(t)\colon t\in\mathbb{R}\}$ as
\[
\mathbb{A}_0(t) = \bigcup_{i = 1}^r W^u(e_i^{\ast})(t) \quad \mbox{for all} \quad 
t\in\mathbb{R},
\]
and the proof is complete.
\end{proof}

\section{The linearized system}\label{linearized system}

Consider the linearization of the problem \eqref{edp01}-\eqref{cond01} around an equilibrium point $e_j^{\ast} \in \mathcal{E}^{\ast}$, $1\leq j\leq r$, which is given by the system
\[
\begin{cases}
\tilde{u}_{tt} + (-\Delta)\tilde{u} + \tilde{u} + \eta(-\Delta)^{1\slash 2}\tilde{u}_t + 
a_{\epsilon}(t)(-\Delta)^{1\slash 2}\tilde{v}_t - f^{\prime}(u_j^{\ast}) \tilde{u} 
= h(\tilde{u}), & \mbox{in} \quad \Omega, \\
\tilde{v}_{tt} + (-\Delta)\tilde{v} + \eta(-\Delta)^{1\slash 2}\tilde{v}_t - 
a_{\epsilon}(t)(-\Delta)^{1\slash 2}\tilde{u}_t = 0, & \mbox{in} \quad \Omega, \\
\tilde{u} = \tilde{v} = 0, & \mbox{on} \quad \partial\Omega, \\
\end{cases}
\]
where
\[
h(\tilde{u}) = f(\tilde{u} + u_j^{\ast}) - f(u_j^{\ast}) - f^{\prime}(u_j^{\ast}) \tilde{u}.
\]
This system can be rewritten in the following abstract form

\begin{equation}\label{linearized abstract ode}
\tilde{W}_t + \tilde{\mathcal{A}}_{\epsilon}(t) \tilde{W} = 
\tilde{F}(\tilde{W}), \quad t > \tau \in\mathbb{R},
\end{equation}
where
\begin{equation}\label{linearized abstract operator}
\tilde{\mathcal{A}}_{\epsilon}(t) \, \!\!=\!\!
\begin{bmatrix}
0     & -I                           & 0 & 0 \\
A+I - f^{\prime}(u_j^{\ast}) & \eta A^{1\slash 2} & 0 & a_{\epsilon}(t) A^{1\slash 2} \\
0     & 0                            & 0 & -I \\
0     & - a_{\epsilon}(t) A^{1\slash 2} & A & \eta A^{1\slash 2}
\end{bmatrix}
\end{equation}
with
\[
\tilde{W} = \begin{bmatrix} \tilde{u}\\ \tilde{u}_t\\ \tilde{v}\\ \tilde{v}_t\\ \end{bmatrix} 
\quad \mbox{and} \quad 
\tilde{F}(\tilde{W}) = \begin{bmatrix} 0\\ h(\tilde{u})\\ 0\\ 0\\ \end{bmatrix}.
\]

We may write $\tilde{\mathcal{A}}_{\epsilon}(t) = \mathcal{A}_{\epsilon}(t) + B$, where $\mathcal{A}_{\epsilon}(t)$ is defined by \eqref{abstract operator}-\eqref{domain abstract operator} and $B$ is given by
\begin{equation}\label{op B}
B \, \!\!=\!\!
\begin{bmatrix}
0     & 0                           & 0 & 0 \\
-f^{\prime}(u_j^{\ast}) & 0 & 0 & 0 \\
0     & 0                            & 0 & 0 \\
0     & 0 & 0 & 0
\end{bmatrix},
\end{equation}
with $D(B) = D(\mathcal{A}_{\epsilon}(t)) = Y_1$.

\begin{proposition}\label{bound op B}
The operator $B$, defined in \eqref{op B}, satisfies
\[
\Vert B \Vert_{\mathcal{L}(Y_0)} \leq 
c\left(1 + \max\limits_{1\leq i\leq r} \Vert e_i^{\ast} \Vert_{Y_0}^{\rho - 1} \right), 
\]
where $c>0$ is a constant.
\end{proposition}

\begin{proof} For a given $\vartheta =\begin{bmatrix} u & v & w & z\end{bmatrix}^T \in Y_0$ and recalling that 
 $e_j^{\ast} = \begin{bmatrix} u_j^{\ast} & 0 & 0 & 0\end{bmatrix}^T\in \mathcal{E}^{\ast},$  we have
\begin{equation}\label{bound B aux 01}
\begin{split}
&\Vert B\vartheta \Vert_{Y_0} \, \!\!=\!\! \,
\left\Vert
\begin{bmatrix}
0 \\
-f^{\prime}(u_j^{\ast}) u\\
0 \\
0
\end{bmatrix}
\right\Vert_{Y_0} 
= \Vert f^{\prime}(u_j^{\ast}) u \Vert_{L^2(\Omega)} \\
&\leq \Vert f^{\prime}(u_j^{\ast}) \Vert_{L^{\frac{2\rho}{\rho - 1}}(\Omega)} 
\Vert u \Vert_{L^{2\rho}(\Omega)} \leq \tilde{c} 
\Vert f^{\prime}(u_j^{\ast}) \Vert_{L^{\frac{2\rho}{\rho - 1}}(\Omega)} 
\Vert u \Vert_{X^{1\slash 2}} \\
&\leq \tilde{c} 
\Vert f^{\prime}(u_j^{\ast}) \Vert_{L^{\frac{2\rho}{\rho - 1}}(\Omega)} 
\Vert \vartheta \Vert_{Y_0},
\end{split}
\end{equation}
where we have used the H\"{o}lder inequality with $\frac{\rho - 1}{2\rho} + \frac{1}{2\rho} = \frac12$, and the embedding $X^{1\slash 2} \hookrightarrow L^{2\rho}(\Omega)$ (as condition $ \frac{N-1}{N-2} \leq \rho < \frac{N}{N-2}$ holds). On the other hand, we can use the growth condition \eqref{Gcondition} to obtain 
\begin{equation}\label{bound B aux 02}
\begin{split}
&\Vert f^{\prime} (u_j^{\ast}) \Vert_{L^{\frac{2\rho}{\rho - 1}}(\Omega)} \leq 
\left( \int_{\Omega} 
[c(1 + \vert u_j^{\ast} \vert^{\rho - 1}) ]^{\frac{2\rho}{\rho - 1}} \, dx 
\right)^{\frac{\rho - 1}{2\rho}} \\
&\leq c_1\left( \vert\Omega\vert + 
\int_{\Omega}\vert u_j^{\ast} \vert^{2\rho} \, dx 
\right)^{\frac{\rho - 1}{2\rho}} \\
&\leq c_2\left( \vert\Omega\vert^{\frac{\rho - 1}{2\rho}} + 
\Vert u_j^{\ast} \Vert_{L^{2\rho}(\Omega)}^{\rho - 1} \right) 
\leq c_3 \left( 1 + \Vert u_j^{\ast} \Vert_{L^{2\rho}(\Omega)}^{\rho - 1} \right) \\
&\leq c_4 \left( 1 + \Vert u_j^{\ast} \Vert_{X^{1\slash 2}}^{\rho - 1} \right) 
= c_4 \left( 1 + \Vert e_j^{\ast} \Vert_{Y_0}^{\rho - 1} \right).
\end{split}
\end{equation}

By \eqref{bound B aux 01} and \eqref{bound B aux 02}, we obtain
\[
\Vert B\vartheta \Vert_{Y_0} \leq 
c_5 \left( 1 + \Vert e_j^{\ast} \Vert_{Y_0}^{\rho - 1} \right) \Vert \vartheta \Vert_{Y_0} \leq 
c_5 \left(1 + \max\limits_{1\leq i\leq r} \Vert e_i^{\ast} \Vert_{Y_0}^{\rho - 1} \right) 
\Vert \vartheta \Vert_{Y_0},
\]
which yields 
$
\Vert B \Vert_{\mathcal{L}(Y_0)} \leq 
c_5 \left(1 + \max\limits_{1\leq i\leq r} \Vert e_i^{\ast} \Vert_{Y_0}^{\rho - 1} \right),
$ 
with $c_5 >0$ being a constant, which proves the result.
\end{proof}

\begin{proposition}\label{convergence inverses}
For each fixed $t\in\mathbb{R}$, there holds 
$
\Vert \tilde{\mathcal{A}}_{\epsilon}^{-1}(t) - 
\tilde{\mathcal{A}}_0^{-1}(t) \Vert_{\mathcal{L}(Y_0)} \rightarrow 0
$
as $\epsilon\rightarrow 0^{+}$.
\end{proposition}

\begin{proof} Note that
\begin{equation}\label{formula inversa A}
\small
\tilde{\mathcal{A}}_{\epsilon}^{-1}(t) =
\begin{bmatrix}
\eta A^{1\slash 2} (A+I - f^{\prime}(u_j^{\ast}))^{-1} & (A+I - f^{\prime}(u_j^{\ast}))^{-1} & a_{\epsilon}(t)A^{1\slash 2} (A+I - f^{\prime}(u_j^{\ast}))^{-1} & 0 \\
-I                                           & 0                & 0                                          & 0 \\
- a_{\epsilon}(t) A^{-1\slash 2}                 & 0               & \eta A^{-1\slash 2}            & A^{-1} \\
0                                            & 0               & -I                                         & 0
\end{bmatrix}. 
\end{equation}

Given $\vartheta =\begin{bmatrix} u & v & w & z\end{bmatrix}^T \in Y_0$, we have
\begin{equation}\label{conv inv aux 01}
\begin{split}
&\Vert (\tilde{\mathcal{A}}_{\epsilon}^{-1}(t) - 
\tilde{\mathcal{A}}_0^{-1}(t) )\vartheta \Vert_{Y_0} \\
&= \left\Vert 
\begin{bmatrix}
[a_{\epsilon}(t) - a_0(t)] A^{1\slash 2} (A+I - f^{\prime}(u_j^{\ast}))^{-1} w\\
0\\ [a_0(t) - a_{\epsilon}(t)] A^{-1\slash 2} u\\ 0
\end{bmatrix}
\right\Vert_{Y_0} \\
&= \Vert [a_{\epsilon}(t) - a_0(t)] A^{1\slash 2} (A+I - f^{\prime}(u_j^{\ast}))^{-1} w 
\Vert_{X^{1\slash 2}} + 
\Vert [a_0(t) - a_{\epsilon}(t)] A^{-1\slash 2} u \Vert_{X^{1\slash 2}} \\
&= \vert a_{\epsilon}(t) - a_0(t)\vert \left( 
\Vert A (A+I - f^{\prime}(u_j^{\ast}))^{-1} w \Vert_X + 
\Vert u\Vert_X \right).
\end{split}
\end{equation}

By condition \eqref{hyper condition}, there exists a constant $\kappa_0 >0$ such that
\[
\Vert (A+I - f^{\prime}(u_j^{\ast}))^{-1} \Vert_{\mathcal{L}(X)} \leq \kappa_0.
\]

Following the proof of Proposition \ref{bound op B} (see \eqref{bound B aux 02}), there exists a constant $c>0$ such that
\[
\Vert f^{\prime} (u_j^{\ast}) \Vert_{L^{\frac{2\rho}{\rho - 1}}(\Omega)} \leq 
c \left( 1 + \Vert e_j^{\ast} \Vert_{Y_0}^{\rho - 1} \right) \leq 
c \left(1 + \max\limits_{1\leq i\leq r} \Vert e_i^{\ast} \Vert_{Y_0}^{\rho - 1} \right).
\]

Now, we use the embedding $X^{1\slash 2} \hookrightarrow L^{2\rho}(\Omega)$ (since $ \frac{N-1}{N-2} \leq \rho < \frac{N}{N-2}$) and H\"{o}lder inequality to obtain
\begin{equation}\label{conv inv aux 02}
\begin{split}
&\Vert A (A+I - f^{\prime}(u_j^{\ast}))^{-1} w \Vert_X \\
&= \left\Vert I - 
(I - f^{\prime}(u_j^{\ast}))\left(A+I - f^{\prime}(u_j^{\ast}))^{-1} 
\right)w \right\Vert_X \\
&\leq \Vert w\Vert_X + 
\Vert (I - f^{\prime}(u_j^{\ast})) (A+I - f^{\prime}(u_j^{\ast}))^{-1} w\Vert_X \\
&\leq \lambda_1^{-1\slash 2} \Vert w \Vert_{X^{1\slash 2}} + 
\Vert (A+I - f^{\prime}(u_j^{\ast}))^{-1} \Vert_{\mathcal{L}(X)} 
\Vert (I - f^{\prime}(u_j^{\ast})) w\Vert_X \\
&\leq \lambda_1^{-1\slash 2} \Vert w \Vert_{X^{1\slash 2}} + 
\kappa_0 \left( \Vert w\Vert_X + 
\Vert f^{\prime}(u_j^{\ast}) w\Vert_X \right) \\
&\leq \lambda_1^{-1\slash 2} \Vert w \Vert_{X^{1\slash 2}} + 
\kappa_{0}\lambda_1^{-1\slash 2} \Vert w \Vert_{X^{1\slash 2}} + 
\kappa_{0} \Vert f^{\prime}(u_j^{\ast})w \Vert_{L^2(\Omega)} \\
&\leq (1+\kappa_0) \lambda_1^{-1\slash 2} \Vert w \Vert_{X^{1\slash 2}} + 
\kappa_{0} \Vert f^{\prime}(u_j^{\ast}) \Vert_{L^{\frac{2\rho}{\rho - 1}}(\Omega)} 
\Vert w \Vert_{L^{2\rho}(\Omega)} \\
&\leq (1+\kappa_0) \lambda_1^{-1\slash 2} \Vert w \Vert_{X^{1\slash 2}} + 
\kappa_{0} c \left(1 + \max\limits_{1\leq i\leq r} \Vert e_i^{\ast} \Vert_{Y_0}^{\rho - 1} \right) \tilde{c} \Vert w \Vert_{X^{1\slash 2}} \\
&= \tilde{\kappa} \Vert w \Vert_{X^{1\slash 2}},
\end{split}
\end{equation}
where 
$
\tilde{\kappa} = (1+\kappa_0) \lambda_1^{-1\slash 2} + 
\kappa_{0} c\tilde{c} \left(1 + \max\limits_{1\leq i\leq r} \Vert e_i^{\ast} \Vert_{Y_0}^{\rho - 1} \right),
$ 
with $\tilde{c} >0$ being the embedding constant of $X^{1\slash 2} \hookrightarrow L^{2\rho}(\Omega)$, and $\lambda_1 > 0$ denotes the first eigenvalue of the negative Laplacian operator with homogeneous Dirichlet boundary condition.

Combining \eqref{conv inv aux 01} and \eqref{conv inv aux 02}, we get
\[
\begin{split}
&\Vert (\tilde{\mathcal{A}}_{\epsilon}^{-1}(t) - 
\tilde{\mathcal{A}}_0^{-1}(t) )\vartheta \Vert_{Y_0} 
\leq \vert a_{\epsilon}(t) - a_0(t)\vert \left( 
\tilde{\kappa} \Vert w \Vert_{X^{1\slash 2}} + 
\lambda_1^{-1\slash 2} \Vert u \Vert_{X^{1\slash 2}} \right) \\
&\leq \vert a_{\epsilon}(t) - a_0(t)\vert 
\max\{ \tilde{\kappa}, \lambda_1^{-1\slash 2} \} 
\left( \Vert u \Vert_{X^{1\slash 2}} + \Vert w \Vert_{X^{1\slash 2}} \right) \\
&\leq \vert a_{\epsilon}(t) - a_0(t)\vert 
\max\{ \tilde{\kappa}, \lambda_1^{-1\slash 2} \} 
\left( \Vert u \Vert_{X^{1\slash 2}} + \Vert v\Vert_X + 
\Vert w \Vert_{X^{1\slash 2}} + \Vert z\Vert_X \right) \\
&= \vert a_{\epsilon}(t) - a_0(t)\vert \tilde{\tilde{\kappa}} 
\Vert \vartheta \Vert_{Y_0},
\end{split}
\]
where the constant $\tilde{\tilde{\kappa}} = \max\{ \tilde{\kappa}, \lambda_1^{-1\slash 2} \} >0$ does not depend on $t\in \R$ neither on the parameter $\epsilon$. Therefore,
\[
\Vert \tilde{\mathcal{A}}_{\epsilon}^{-1}(t) - 
\tilde{\mathcal{A}}_0^{-1}(t) \Vert_{\mathcal{L}(Y_0)} \leq 
\tilde{\tilde{\kappa}} 
\Vert a_{\epsilon} - a_0 \Vert_{L^{\infty}(\mathbb{R})} 
\rightarrow 0 \quad \text{as} \quad \epsilon\rightarrow 0^{+},
\]
which proves our result.
\end{proof}

\begin{proposition}\label{convergence operators e and 0}
For each fixed $t,\tau\in\mathbb{R}$, there holds
\[
\Vert [\tilde{\mathcal{A}}_{\epsilon}(t) - 
\tilde{\mathcal{A}}_0(t)] \tilde{\mathcal{A}}^{-1}_\epsilon(\tau) \Vert_{\mathcal{L}(Y_0)} \leqslant \Vert a_\epsilon - a_0 \Vert_{L^{\infty}(\mathbb{R})}.
\]
\end{proposition}
\begin{proof} Given $\vartheta =\begin{bmatrix} u & v & w & z\end{bmatrix}^T \in Y_0$ and using \eqref{formula inversa A}, we have
\[
\begin{split}
\Vert [\tilde{\mathcal{A}}_{\epsilon}(t) - 
\tilde{\mathcal{A}}_0(t)] \tilde{\mathcal{A}}^{-1}_\epsilon(\tau)\vartheta \Vert_{Y_0} & =
 \left\Vert 
\begin{bmatrix}
0 \\
-[a_{\epsilon}(t) - a_0(t)] A^{\frac12}w  \\
0 \\ 
[a_{\epsilon}(t) - a_0(t)] A^{\frac12} u 
\end{bmatrix}
\right\Vert_{Y_0} \\
&= 
\Vert 
[a_{\epsilon}(t) - a_0(t)] A^{\frac12}w
\Vert_{X}
+ \Vert -[a_{\epsilon}(t) - a_0(t)] A^{\frac12}u\Vert_{X}   \\
&\leqslant 
\Vert a_\epsilon - a_0 \Vert_{L^{\infty}(\mathbb{R})} \|\vartheta \|_{Y_0}.
\end{split}
 \]

\end{proof}

\subsection{Uniform sectorial estimates and convergence of the linear semigroups}

Consider the sector $\Gamma$, whose angle is greater than $\frac{\pi}{2}$, given in Lemma \ref{Lemma 3.2 aux AL1}.

\begin{lemma}\label{sec lemma 03}
There exists $R_0 >0$ sufficiently large such that, if $\mu\in\Gamma$ satisfies $\vert\mu\vert\geq R_0$, then the operator $\tilde{\mathcal{A}}_{\epsilon}(t) = \mathcal{A}_{\epsilon}(t) + B$, where $B$ is given by \eqref{op B}, satisfies the estimate
\[
\Vert (\mu I + \tilde{\mathcal{A}}_{\epsilon}(t))^{-1} \Vert_{\mathcal{L}(Y_0)} \leq 
\frac{\overline{C}}{\vert\mu\vert},
\]
where the constant $\overline{C}>0$ is independent of both $t\in\mathbb{R}$ and $\epsilon\in [0, 1]$.
\end{lemma}

\begin{proof}
Note that it follows from \eqref{lim A(I)} that
\[
\begin{split}
\Vert \mu (\mu I + \mathcal{A}_{\epsilon}(t))^{-1} \Vert_{\mathcal{L}(Y_0)} &= \Vert I - \mathcal{A}_{\epsilon}(t) (\mu I + \mathcal{A}_{\epsilon}(t))^{-1} \Vert_{\mathcal{L}(Y_0)} \\
&\leq 1 + \Vert \mathcal{A}_{\epsilon}(t) (\mu I + \mathcal{A}_{\epsilon}(t))^{-1} \Vert_{\mathcal{L}(Y_0)} \leq 2 + \tilde{C} = C_1,
\end{split}
\]
that is,
\begin{equation}\label{sec est aux 01}
\Vert \mu (\mu I + \mathcal{A}_{\epsilon}(t))^{-1} \Vert_{\mathcal{L}(Y_0)} \leq C_1, \quad \mu \in \Gamma.
\end{equation}

Moreover, since
\[
\begin{split}
\mu I + \tilde{\mathcal{A}}_{\epsilon}(t) &= (\mu I + \mathcal{A}_{\epsilon}(t)) + B \\
&= (\mu I + \mathcal{A}_{\epsilon}(t)) + B(\mu I + \mathcal{A}_{\epsilon}(t))^{-1}(\mu I + \mathcal{A}_{\epsilon}(t)) \\
&= [I + B(\mu I + \mathcal{A}_{\epsilon}(t))^{-1}] (\mu I + \mathcal{A}_{\epsilon}(t)),
\end{split}
\]
we have
\begin{equation}\label{sec est aux 02}
(\mu I + \tilde{\mathcal{A}}_{\epsilon}(t))^{-1} = 
(\mu I + \mathcal{A}_{\epsilon}(t))^{-1} 
[I + B(\mu I + \mathcal{A}_{\epsilon}(t))^{-1}]^{-1}, \quad \mu \in \Gamma.
\end{equation}

On the other hand, by Proposition \ref{bound op B}, there is a constant $c>0$ such that
\[
\Vert B \Vert_{\mathcal{L}(Y_0)} \leq 
c\left(1 + \max\limits_{1\leq i\leq r} \Vert e_i^{\ast} \Vert_{Y_0}^{\rho - 1} \right), 
\]
that is, $B$ is a bounded linear operator from $Y_0$ into itself and, therefore, there exists a constant $k_0 >0$ such that 
$
\Vert -Bx\Vert_{Y_0} \leq k_0\Vert x\Vert_{Y_0} 
$
for all $x\in D(B) = D(\mathcal{A}_{\epsilon}(t)) = Y_1$.

Now, given $\alpha\in (0, 1),$ take $R_0 >0$ such that $1 - \frac{k_0 C_1}{R_0} \geq\alpha$ and let $\xi >0$ be such that $\xi (1+\tilde{C}) <\alpha$. 
In this way, we have
\[
1 - \xi (1+\tilde{C}) - \frac{k_0 C_1}{R_0} \geq \alpha - \xi (1+\tilde{C}) >0.
\]
Furthermore, we also have
\[
\Vert -Bx\Vert_{Y_0} \leq k_0\Vert x\Vert_{Y_0} + \xi\Vert \mathcal{A}_{\epsilon}(t)x \Vert_{Y_0}
\]
for all $x\in Y_1$. Thus, using \eqref{lim A(I)} and \eqref{sec est aux 01}, we obtain
\[
\begin{split}
\Vert -B(\mu I + \mathcal{A}_{\epsilon}(t))^{-1} \Vert_{\mathcal{L}(Y_0)} &\leq \xi\Vert \mathcal{A}_{\epsilon}(t) (\mu I + \mathcal{A}_{\epsilon}(t))^{-1} \Vert_{\mathcal{L}(Y_0)} + 
k_0\Vert (\mu I + \mathcal{A}_{\epsilon}(t))^{-1} \Vert_{\mathcal{L}(Y_0)} \\
&\leq \xi (1+\tilde{C}) + \frac{k_0 C_1}{\vert\mu\vert} 
\end{split}
\]
and, consequently,
\[
1 - \Vert -B(\mu I + \mathcal{A}_{\epsilon}(t))^{-1} \Vert_{\mathcal{L}(Y_0)} 
\geq 1 - \xi (1+\tilde{C}) - \frac{k_0 C_1}{\vert\mu\vert}, \quad \mu \in \Gamma, \mu \neq 0.
\]

An application of the Neumann series leads to
\begin{equation}\label{sec est aux 03}
\begin{split}
\Vert [I + B(\mu I + \mathcal{A}_{\epsilon}(t))^{-1}]^{-1} \Vert_{\mathcal{L}(Y_0)} 
&\leq 
\frac{1}{1 - \Vert -B(\mu I + \mathcal{A}_{\epsilon}(t))^{-1} \Vert_{\mathcal{L}(Y_0)}} \\
&\leq \left( 1 - \xi (1+\tilde{C}) - \frac{k_0 C_1}{\vert\mu\vert} \right)^{-1}. 
\end{split}
\end{equation}

Now, if $\mu \in \Gamma$ with $\vert\mu\vert\geq R_0$, taking $C_2 = 1 - \xi (1+\tilde{C}) - \frac{k_0 C_1}{R_0} >0$, we derive
\begin{equation}\label{sec est aux 04}
\left( 1 - \xi (1+\tilde{C}) - \frac{k_0 C_1}{\vert\mu\vert} \right)^{-1} \leq \frac{1}{C_2}
\end{equation}
and, combining \eqref{sec est aux 01}, \eqref{sec est aux 02}, \eqref{sec est aux 03} and \eqref{sec est aux 04}, it follows that
\begin{equation}\label{sec estimate 0300}
\begin{split}
\Vert (\mu I + \tilde{\mathcal{A}}_{\epsilon}(t))^{-1} \Vert_{\mathcal{L}(Y_0)} &\leq\Vert (\mu I + \mathcal{A}_{\epsilon}(t))^{-1} \Vert_{\mathcal{L}(Y_0)} 
\Vert [I + B(\mu I + \mathcal{A}_{\epsilon}(t))^{-1}]^{-1} \Vert_{\mathcal{L}(Y_0)} \\
&\leq \frac{C_1}{\vert\mu\vert} \frac{1}{C_2} = \frac{C_3}{\vert\mu\vert},
\end{split}
\end{equation}
where the constant $C_3 = \frac{C_1}{C_2} >0$ is independent of $t\in\mathbb{R}$ and $\epsilon\in [0, 1]$. This concludes the proof.
\end{proof}

\begin{remark}\label{remark sector} 
The estimate obtained in \eqref{sec estimate 0300} holds outside a ball of radius $R_0 >0$. 
Now, let $a > R_0$, $\phi_0\in (\pi\slash 2, \pi)$ be the angle of the sector $\Gamma$ and consider the sector
\[
\Sigma_a = \{ \lambda\in\mathbb{C}\colon \vert \arg (\lambda - a)\vert \leq \phi_0, \, \lambda \neq a\}.
\]
If $\mu\in \Sigma_a$, then by \eqref{sec estimate 0300}, we obtain
\[
\begin{split}
\Vert (\mu I + \tilde{\mathcal{A}}_{\epsilon}(t))^{-1} \Vert_{\mathcal{L}(Y_0)} 
&\leq \frac{C_3}{\vert\mu\vert} 
\leq \frac{C_3}{\vert\mu - a\vert} \left(1 + \frac{\vert a\vert}{\vert\mu\vert} \right) \leq \frac{C_3}{\vert\mu - a\vert} \left(1 + \frac{\vert a\vert}{R_0} \right) 
\end{split}
\]
provided $\vert\mu\vert\geq R_0$, that is, 
\[
\Vert (\mu I + \tilde{\mathcal{A}}_{\epsilon}(t))^{-1} \Vert_{\mathcal{L}(Y_0)} 
\leq \frac{C_4}{\vert\mu - a\vert} 
\quad \mbox{for all} \quad \mu\in \Sigma_a,
\]
where the constant $C_4 = C_3\left(1 + \frac{\vert a\vert}{R_0} \right)  >0$ is independent of $t\in\mathbb{R}$ and $\epsilon\in [0, 1]$.
\end{remark}



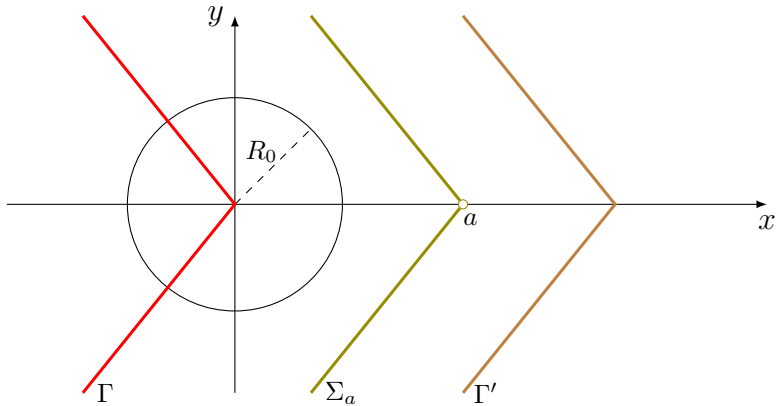
\begin{center}
\begin{tikzpicture}
\draw[-latex] (-3,0)--(7,0) node[below]{$x$};
\draw[-latex] (0,-2.5)--(0,2.5) node[left]{$y$};
\draw (0,0) circle (1.4142cm);
\draw[dashed] (0,0) -- (1,1);
\node at (0.35,0.7) {\footnotesize $R_0$};
\draw[color=red,line width=0.4mm] (-2,-2.5) -- (0,0) -- (-2,2.5);
\draw[color=olive,line width=0.4mm] (1,-2.5) -- (3,0) -- (1,2.5);
\draw[color=brown,line width=0.4mm] (3,-2.5) -- (5,0) -- (3,2.5);
\node at (3.1,-0.2) {\footnotesize $a$};
\node at (1.4,-2.5) {\footnotesize $\Sigma_a$};
\node at (3.3,-2.5) {\footnotesize $\Gamma'$};
\node at (-1.7,-2.5) {\footnotesize $\Gamma$};
\filldraw[white] (3,0) circle (0.6mm);
\draw[olive] (3,0) circle (0.6mm);
\end{tikzpicture}
\end{center}

For each $\epsilon\in [0, 1],$ let 
$
\{ e^{-s \tilde{\mathcal{A}}_{\epsilon}(t)} \colon s\geq 0\} \subset \mathcal{L}(Y_0)
$
be the analytic semigroup generated by the operator $-\tilde{\mathcal{A}}_{\epsilon}(t)$, for each $t\in\mathbb{R}$ fixed. In Proposition \ref{convergence linear semigroups} below we present a result on convergence for this family of semigroups as $\epsilon\rightarrow 0^{+}$.

\begin{proposition}\label{convergence linear semigroups}
For $\epsilon\in [0, 1],$ the family of semigroups 
$
\{ e^{-s \tilde{\mathcal{A}}_{\epsilon}(t)} \colon s\geq 0\}
$
satisfies
\[
\Vert e^{-s \tilde{\mathcal{A}}_{\epsilon}(t)} - e^{-s \tilde{\mathcal{A}}_0(t)} 
\Vert_{\mathcal{L}(Y_0)} \rightarrow 0 
\quad \mbox{as} \quad \epsilon\rightarrow 0^{+}
\]
uniformly for $s$ on compact subsets of $\mathbb{R}_+$.
\end{proposition}

\begin{proof}
Let $t\in\mathbb{R}$ be fixed and arbitrary. 
According to \cite[Theorem 1.3.4]{Henry}, for each $\epsilon\in [0, 1],$ the semigroup 
$
\{ e^{-s \tilde{\mathcal{A}}_{\epsilon}(t)} \colon s\geq 0\}
$
can be written as
\begin{equation}\label{escrita da exponencial}
e^{-s \tilde{\mathcal{A}}_{\epsilon}(t)} = \frac{1}{2\pi i} \int_{\Gamma'} e^{\mu s} 
(\mu I + \tilde{\mathcal{A}}_{\epsilon}(t))^{-1} \, d\mu, \quad s\geq 0,
\end{equation}
where the curve $\Gamma'$ is a contour in $\rho(-\tilde{\mathcal{A}}_{\epsilon}(t))$, with $\arg\mu \rightarrow \pm\theta$ as $\vert\mu\vert \rightarrow \infty$, for some $\theta \in (\pi\slash 2, \pi)$. In addition to the previous conditions, let us assume, without loss of generality, $\Gamma'$ is contained in the sector $\Sigma_a$ given by
\[
\Sigma_a = \{ \lambda\in\mathbb{C}\colon  \vert \arg (\lambda - a)\vert \leq\theta, \, \lambda \neq a\},
\]
where $a > R_0$ and $R_0 >0$ is given as in Lemma \ref{sec lemma 03} (see also Remark \ref{remark sector}). 
At first place, we have
\begin{equation}\label{norm dif semigroups}
\begin{split}
&\Vert e^{-s \tilde{\mathcal{A}}_{\epsilon}(t)} - e^{-s \tilde{\mathcal{A}}_0(t)} 
\Vert_{\mathcal{L}(Y_0)} \\
&= \left\Vert \frac{1}{2\pi i} \int_{\Gamma'} e^{\mu s} 
(\mu I + \tilde{\mathcal{A}}_{\epsilon}(t))^{-1} \, d\mu - \frac{1}{2\pi i} \int_{\Gamma'} 
e^{\mu s} (\mu I + \tilde{\mathcal{A}}_0(t))^{-1} \, d\mu \right\Vert_{\mathcal{L}(Y_0)} \\
&\leq \frac{1}{2\pi}\int_{\Gamma'} \vert e^{\mu s}\vert \Vert 
(\mu I + \tilde{\mathcal{A}}_{\epsilon}(t))^{-1} - 
(\mu I + \tilde{\mathcal{A}}_0(t))^{-1} \Vert_{\mathcal{L}(Y_0)} \, d\mu.
\end{split}
\end{equation}

Now, in order to deal with the resolvent operators in \eqref{norm dif semigroups}, we first write
\begin{equation}\label{dif resolvent}
\begin{split}
&\Vert (\mu I + \tilde{\mathcal{A}}_{\epsilon}(t))^{-1} - 
(\mu I + \tilde{\mathcal{A}}_0(t))^{-1} \Vert_{\mathcal{L}(Y_0)} \\
&= \Vert (\mu I + \tilde{\mathcal{A}}_{\epsilon}(t))^{-1} 
[ (\mu I + \tilde{\mathcal{A}}_0(t)) - (\mu I + \tilde{\mathcal{A}}_{\epsilon}(t)) ] 
(\mu I + \tilde{\mathcal{A}}_0(t))^{-1} \Vert_{\mathcal{L}(Y_0)} \\
&= \Vert (\mu I + \tilde{\mathcal{A}}_{\epsilon}(t))^{-1} \tilde{\mathcal{A}}_{\epsilon}(t) 
[ \tilde{\mathcal{A}}_{\epsilon}^{-1}(t) - \tilde{\mathcal{A}}_0^{-1}(t) ] 
\tilde{\mathcal{A}}_0(t) (\mu I + \tilde{\mathcal{A}}_0(t))^{-1} \Vert_{\mathcal{L}(Y_0)} \\
&\leq \Vert (\mu I + \tilde{\mathcal{A}}_{\epsilon}(t))^{-1} \tilde{\mathcal{A}}_{\epsilon}(t) \Vert_{\mathcal{L}(Y_0)} 
\Vert \tilde{\mathcal{A}}_{\epsilon}^{-1}(t) - \tilde{\mathcal{A}}_0^{-1}(t) \Vert_{\mathcal{L}(Y_0)} 
\Vert \tilde{\mathcal{A}}_0(t) (\mu I + \tilde{\mathcal{A}}_0(t))^{-1} \Vert_{\mathcal{L}(Y_0)}.
\end{split}
\end{equation}

On the other hand, using Lemma \ref{sec lemma 03}, we have
\begin{equation}\label{bound aux 01}
\begin{split}
&\Vert (\mu I + \tilde{\mathcal{A}}_{\epsilon}(t))^{-1} \tilde{\mathcal{A}}_{\epsilon}(t) \Vert_{\mathcal{L}(Y_0)} \leq 1 + \vert\mu\vert \Vert (\mu I + \tilde{\mathcal{A}}_{\epsilon}(t))^{-1} \Vert_{\mathcal{L}(Y_0)} \leq 1 + \overline{C},
\end{split}
\end{equation}
where $\hat{C} := 1 + \overline{C} >0$ is independent of $t\in\mathbb{R}$ and $\epsilon\in [0, 1]$. Analogously, we also have
\begin{equation}\label{bound aux 02}
\Vert \tilde{\mathcal{A}}_0(t) (\mu I + \tilde{\mathcal{A}}_0(t))^{-1} \Vert_{\mathcal{L}(Y_0)} \leq \hat{C}.
\end{equation}

Thus, combining \eqref{dif resolvent}, \eqref{bound aux 01} and \eqref{bound aux 02}, together with Proposition \ref{convergence inverses}, we get
\[
\Vert (\mu I + \tilde{\mathcal{A}}_{\epsilon}(t))^{-1} - 
(\mu I + \tilde{\mathcal{A}}_0(t))^{-1} \Vert_{\mathcal{L}(Y_0)} \rightarrow 0 
\quad \mbox{as} \quad \epsilon\rightarrow 0^{+}
\]
and, consequently, it follows by \eqref{norm dif semigroups} and the Dominated Convergence Theorem that
\[
\begin{split}
&\Vert e^{-s \tilde{\mathcal{A}}_{\epsilon}(t)} - e^{-s \tilde{\mathcal{A}}_0(t)} 
\Vert_{\mathcal{L}(Y_0)} \leq \frac{1}{2\pi}\int_{\Gamma'} \vert e^{\mu s}\vert \Vert 
(\mu I + \tilde{\mathcal{A}}_{\epsilon}(t))^{-1} - 
(\mu I + \tilde{\mathcal{A}}_0(t))^{-1} \Vert_{\mathcal{L}(Y_0)} \, d\mu 
\xrightarrow{\epsilon \rightarrow 0^+} 0
\end{split}
\]
uniformly for $s\geq 0$ on compact subsets of $\mathbb{R}_+$, which yields the desired result.
\end{proof}


\section{Exponential dichotomy and lower semicontinuity of pullback attractors}\label{sub exponential dichotomy}

\begin{proposition}\label{uniform holder A til}
The map $\mathbb{R}\ni t\mapsto \tilde{\mathcal{A}}_{\epsilon}(t)$ is uniformly H\"{o}lder continuous in $Y_0,$ for each $\epsilon\in [0, 1],$ that is, for any $T>0$
\[
\Vert [ \tilde{\mathcal{A}}_{\epsilon}(t) - \tilde{\mathcal{A}}_{\epsilon}(s) ] 
\tilde{\mathcal{A}}_{\epsilon}^{-1}(\tau) \Vert_{\mathcal{L}(Y_0)} \leq 
C\vert t-s\vert^{\beta}
\]
for every $t, s, \tau \in [-T, T]$, where $C > 0$ and $0 < \beta \leq 1$ $($both independent of $T$ and $\epsilon$$)$ come from condition \ref{hol-a}.
\end{proposition}

\begin{proof}
Let $T>0$ be fixed and arbitrary, and let $t, s, \tau \in [-T, T]$. 
Since 
\[
\tilde{\mathcal{A}}_{\epsilon}(t) - \tilde{\mathcal{A}}_{\epsilon}(s) \, \!\!=\!\! \,
[a_{\epsilon}(t) - a_{\epsilon}(s)]
\begin{bmatrix}
0 & 0 & 0 & 0 \\
0 & 0 & 0 & A^{1\slash 2} \\
0 & 0 & 0 & 0 \\
0 & -A^{1\slash 2} & 0 & 0
\end{bmatrix} 
\]
and
\[
\small
\tilde{\mathcal{A}}_{\epsilon}^{-1}(\tau) =
\begin{bmatrix}
\eta A^{1\slash 2} (A+I - f^{\prime}(u_j^{\ast}))^{-1} & (A+I - f^{\prime}(u_j^{\ast}))^{-1} & a_{\epsilon}(\tau)A^{1\slash 2} (A+I - f^{\prime}(u_j^{\ast}))^{-1} & 0 \\
-I                                           & 0                & 0                                          & 0 \\
- a_{\epsilon}(\tau) A^{-1\slash 2}                 & 0               & \eta A^{-1\slash 2}            & A^{-1} \\
0                                            & 0               & -I                                         & 0
\end{bmatrix}, 
\]
we obtain
\[
\begin{split}
\left\Vert [ \tilde{\mathcal{A}}_{\epsilon}(t) - \tilde{\mathcal{A}}_{\epsilon}(s) ] 
\tilde{\mathcal{A}}_{\epsilon}^{-1}(\tau) 
\begin{bmatrix}
u\\ v\\ w\\ z
\end{bmatrix}
\right\Vert_{Y_0} 
&= \left\Vert [a_{\epsilon}(t) - a_{\epsilon}(s)] 
\begin{bmatrix}
0 & 0 & 0 & 0 \\
0 & 0 & -A^{1\slash 2} & 0 \\
0 & 0 & 0 & 0 \\
A^{1\slash 2} & 0 & 0 & 0
\end{bmatrix} 
\begin{bmatrix}
u\\ v\\ w\\ z
\end{bmatrix}
\right\Vert_{Y_0} \\
&= \vert a_{\epsilon}(t) - a_{\epsilon}(s) \vert 
\left\Vert
\begin{bmatrix}
0\\ -A^{1\slash 2}w\\ 0\\ A^{1\slash 2}u
\end{bmatrix}
\right\Vert_{Y_0} \\
&= \vert a_{\epsilon}(t) - a_{\epsilon}(s) \vert 
\left( \left\Vert -A^{1\slash 2}w \right\Vert_X 
+ \left\Vert A^{1\slash 2}u \right\Vert_X \right) \\
&\leq C\vert t-s\vert^{\beta} 
\left( \left\Vert u\right\Vert_{X^{1\slash 2}} + \left\Vert v\right\Vert_X 
+ \left\Vert w\right\Vert_{X^{1\slash 2}} 
+ \left\Vert z\right\Vert_X \right) \\
&= C\vert t-s\vert^{\beta} 
\left\Vert
\begin{bmatrix}
u\\ v\\ w\\ z
\end{bmatrix} 
\right\Vert_{Y_0}
\end{split}
\]
for every $[u\,\, v\,\, w\,\, z]^T \in Y_0$,
where we have used the H\"{o}lder continuity condition of the function $a_{\epsilon}$ uniformly in $\epsilon$. 
Thus, it follows that
\[
\Vert [\tilde{\mathcal{A}}_{\epsilon}(t) - \tilde{\mathcal{A}}_{\epsilon}(s)] 
\tilde{\mathcal{A}}_{\epsilon}^{-1}(\tau) 
\Vert_{\mathcal{L}(Y_0)} \leq C\vert t-s\vert^{\beta}
\]
for every $t, s, \tau \in [-T, T],$ and the result is proved.
\end{proof}

The next result uses Proposition \ref{convergence linear semigroups} to obtain the convergence of the family of linear semigroups in the norm topology of $\mathcal{L}(Y_1)$, which will be helpful in the proof of Proposition \ref{convergence processes linearization}.

\begin{lemma}\label{conv semigroups in Y0 Y1}
For $s\geq\tau\in\mathbb{R}$, there holds
\begin{equation}\label{conv linear process 070}
\Vert e^{-(s-\tau) \tilde{\mathcal{A}}_{\epsilon}(\tau)} - e^{-(s-\tau) \tilde{\mathcal{A}}_0(\tau)} \Vert_{\mathcal{L}(Y_0, Y_1)} 
\rightarrow 0 \quad \mbox{as} \quad \epsilon\rightarrow 0^{+}.
\end{equation}
\end{lemma}
\begin{proof}
Note that
\[
\begin{aligned}
\tilde{\mathcal{A}}_{\epsilon}(\tau) \Big[e^{-(s-\tau) \tilde{\mathcal{A}}_{\epsilon}(\tau)} - e^{-(s-\tau) \tilde{\mathcal{A}}_0(\tau)}\Big] &= 
\tilde{\mathcal{A}}_{\epsilon}(\tau) e^{-\frac{(s-\tau)}{2} \tilde{\mathcal{A}}_{\epsilon}(\tau)} \Big[ e^{-\frac{(s-\tau)}{2} \tilde{\mathcal{A}}_{\epsilon}(\tau)} - e^{-\frac{(s-\tau)}{2} \tilde{\mathcal{A}}_{0}(\tau)}\Big] \\
& \quad + e^{-\frac{(s-\tau)}{2} \tilde{\mathcal{A}}_{\epsilon}(\tau)} \Big[ \tilde{\mathcal{A}}_{\epsilon}(\tau) - \tilde{\mathcal{A}}_{0}(\tau)\Big] e^{-\frac{(s-\tau)}{2} \tilde{\mathcal{A}}_{0}(\tau)} \\
& \quad + \Big[e^{-\frac{(s-\tau)}{2} \tilde{\mathcal{A}}_{\epsilon}(\tau)} - e^{-\frac{(s-\tau)}{2} \tilde{\mathcal{A}}_{0}(\tau)} \Big] \tilde{\mathcal{A}}_{0}(\tau) e^{-\frac{(s-\tau)}{2} \tilde{\mathcal{A}}_{0}(\tau)}\\
& \quad + \Big[\tilde{\mathcal{A}}_{0}(\tau) - \tilde{\mathcal{A}}_{\epsilon}(\tau) \Big] e^{-(s-\tau) \tilde{\mathcal{A}}_{0}(\tau)}.
%
%
\end{aligned}
\]

Due to Lemma \ref{sec lemma 03} and \eqref{escrita da exponencial}, we have the following estimate
\[
\begin{split}
\Vert e^{-(s-\tau) \tilde{\mathcal{A}}_{\epsilon}(\tau)}x \Vert_{Y_1} &= 
\Vert \tilde{\mathcal{A}}_{\epsilon}(\tau) 
e^{-(s-\tau) \tilde{\mathcal{A}}_{\epsilon}(\tau)}x \Vert_{Y_0} \leq M_0 (s-\tau)^{-1} e^{-\omega (s-\tau)} \Vert x\Vert_{Y_0}
\end{split}
\]
for every $x\in Y_0$, where the constants $M_0, \omega >0$ are both independent of the parameters $t$ and $\epsilon$, that is,
\begin{equation}\label{conv linear process 09 extra}
\Vert e^{-(s-\tau) \tilde{\mathcal{A}}_{\epsilon}(\tau)} \Vert_{\mathcal{L}(Y_0, Y_1)} 
\leq M_0 (s-\tau)^{-1} e^{-\omega (s-\tau)}.
\end{equation}

Using Propositions \ref{convergence operators e and 0} and \ref{convergence linear semigroups}, condition \eqref{norm function a} and \eqref{conv linear process 09 extra} the result is complete.
\end{proof}

\begin{proposition}\label{convergence processes linearization}
For $\epsilon\in [0, 1],$ let 
$
\{ \tilde{L}_{\epsilon}(t, \tau) \colon t\geq \tau\in\mathbb{R} \}
$ 
be the linear evolution process in $Y_0$ corresponding to the linearized problem \eqref{linearized abstract ode}. Then, for each $t\geq\tau\in\mathbb{R}$,
\[
\Vert \tilde{L}_{\epsilon}(t, \tau) - \tilde{L}_0(t, \tau) \Vert_{\mathcal{L}(Y_0)} 
\rightarrow 0 \quad \mbox{as} \quad \epsilon\rightarrow 0^{+}.
\]
\end{proposition}

\begin{proof}
According to \cite{Nascimento}, the process $\{ \tilde{L}_{\epsilon}(t, \tau) \colon t\geq \tau\in\mathbb{R} \}$ can be written as
$$
\tilde{L}_{\epsilon}(t, \tau) = e^{-(t-\tau) \tilde{\mathcal{A}}_{\epsilon}(\tau)}
+ \int_{\tau}^t \tilde{L}_{\epsilon}(t, s) [ \tilde{\mathcal{A}}_{\epsilon}(\tau) - \tilde{\mathcal{A}}_{\epsilon}(s) ] e^{-(s-\tau) \tilde{\mathcal{A}}_{\epsilon}(\tau)} \, ds, \quad t\geq \tau \in\mathbb{R},
$$
for each $\epsilon \in [0,1],$ where $\tilde{\mathcal{A}}_{\epsilon}(\tau)$ is given as in \eqref{linearized abstract operator}. 
At first place, we have
\begin{equation}\label{conv linear process 01}
\begin{split}
\Vert \tilde{L}_{\epsilon}(t, \tau) - \tilde{L}_0(t, \tau) \Vert_{\mathcal{L}(Y_0)} &\leq \Vert e^{-(t-\tau) \tilde{\mathcal{A}}_{\epsilon}(\tau)} - e^{-(t-\tau) \tilde{\mathcal{A}}_0(\tau)} \Vert_{\mathcal{L}(Y_0)} \\
&+ \Big\Vert \int_{\tau}^t \tilde{L}_{\epsilon}(t, s) [ \tilde{\mathcal{A}}_{\epsilon}(\tau) - \tilde{\mathcal{A}}_{\epsilon}(s) ] e^{-(s-\tau) \tilde{\mathcal{A}}_{\epsilon}(\tau)} \, ds \\
&\quad\quad - \int_{\tau}^t \tilde{L}_0(t, s) [ \tilde{\mathcal{A}}_0(\tau) - \tilde{\mathcal{A}}_0(s) ] e^{-(s-\tau) \tilde{\mathcal{A}}_0(\tau)} \, ds \Big\Vert_{\mathcal{L}(Y_0)}. 
\end{split}
\end{equation}

By Proposition \ref{convergence linear semigroups}, we have
\begin{equation}\label{conv linear process 02}
\Vert e^{-(t-\tau) \tilde{\mathcal{A}}_{\epsilon}(\tau)} - e^{-(t-\tau) \tilde{\mathcal{A}}_0(\tau)} \Vert_{\mathcal{L}(Y_0)} 
\rightarrow 0 \quad \mbox{as} \quad \epsilon\rightarrow 0^{+}.
\end{equation}

Note that the difference between the integrals in \eqref{conv linear process 01} can be written in the following way

\begin{equation}\label{conv linear process 03}
\small
\begin{split}
&\int_{\tau}^t \tilde{L}_{\epsilon}(t, s) [ \tilde{\mathcal{A}}_{\epsilon}(\tau) - \tilde{\mathcal{A}}_{\epsilon}(s) ] e^{-(s-\tau) \tilde{\mathcal{A}}_{\epsilon}(\tau)} \, ds
- \int_{\tau}^t \tilde{L}_{0}(t, s) [ \tilde{\mathcal{A}}_0(\tau) - \tilde{\mathcal{A}}_{0}(s) ] e^{-(s-\tau) \tilde{\mathcal{A}}_{0}(\tau)} \, ds \\
&= \int_{\tau}^t \tilde{L}_{\epsilon}(t, s) [ \tilde{\mathcal{A}}_{\epsilon}(\tau) - \tilde{\mathcal{A}}_{\epsilon}(s)] (e^{-(s-\tau) \tilde{\mathcal{A}}_{\epsilon}(\tau)} - e^{-(s-\tau) \tilde{\mathcal{A}}_0(\tau)} ) \, ds \\
&\quad + \int_{\tau}^t ( \tilde{L}_{\epsilon}(t, s) - \tilde{L}_0(t, s) ) 
[ \tilde{\mathcal{A}}_{\epsilon}(\tau) - \tilde{\mathcal{A}}_{\epsilon}(s) ] 
e^{-(s-\tau) \tilde{\mathcal{A}}_0(\tau)} \, ds \\
&\quad + \int_{\tau}^t \tilde{L}_0(t, s) [ \tilde{\mathcal{A}}_{\epsilon}(\tau) - \tilde{\mathcal{A}}_{\epsilon}(s) ] e^{-(s-\tau) \tilde{\mathcal{A}}_0(\tau)} \, ds \\
&\quad - \int_{\tau}^t \tilde{L}_0(t, s) [ \tilde{\mathcal{A}}_0(\tau) - \tilde{\mathcal{A}}_0(s) ] e^{-(s-\tau) \tilde{\mathcal{A}}_0(\tau)} \, ds. 
\end{split}
\end{equation}

Now, we are going to analyze each one of the integrals that appear in \eqref{conv linear process 03}, following the order listed below \\

\noindent (A) 
$\displaystyle 
\int_{\tau}^t \tilde{L}_0(t, s) [ \tilde{\mathcal{A}}_{\epsilon}(\tau) - \tilde{\mathcal{A}}_{\epsilon}(s) ] e^{-(s-\tau) \tilde{\mathcal{A}}_0(\tau)} \, ds 
- \int_{\tau}^t \tilde{L}_0(t, s) [ \tilde{\mathcal{A}}_0(\tau) - \tilde{\mathcal{A}}_0(s) ] e^{-(s-\tau) \tilde{\mathcal{A}}_0(\tau)} \, ds
$; \\

\noindent (B) 
$\displaystyle 
\int_{\tau}^t ( \tilde{L}_{\epsilon}(t, s) - \tilde{L}_0(t, s) ) 
[ \tilde{\mathcal{A}}_{\epsilon}(\tau) - \tilde{\mathcal{A}}_{\epsilon}(s) ] 
e^{-(s-\tau) \tilde{\mathcal{A}}_0(\tau)} \, ds
$; \\

\noindent (C) 
$\displaystyle 
\int_{\tau}^t \tilde{L}_{\epsilon}(t, s) [ \tilde{\mathcal{A}}_{\epsilon}(\tau) - \tilde{\mathcal{A}}_{\epsilon}(s)] (e^{-(s-\tau) \tilde{\mathcal{A}}_{\epsilon}(\tau)} - e^{-(s-\tau) \tilde{\mathcal{A}}_0(\tau)} ) \, ds.
$

In fact, for item (A), let us denote 
$
\hat{A} \, \!\!=\!\!
\begin{bmatrix}
0     & 0                           & 0 & 0 \\
0     & 0 & 0 & A^{1\slash 2} \\
0     & 0                            & 0 & 0 \\
0     & -A^{1\slash 2} & 0 & 0
\end{bmatrix}
$
and then note that
\[
\begin{split}
&\int_{\tau}^t \tilde{L}_0(t, s) [ \tilde{\mathcal{A}}_{\epsilon}(\tau) - \tilde{\mathcal{A}}_{\epsilon}(s) ] e^{-(s-\tau) \tilde{\mathcal{A}}_{0}(\tau)} \, ds
- \int_{\tau}^t \tilde{L}_0(t, s) [ \tilde{\mathcal{A}}_{0}(\tau) - \tilde{\mathcal{A}}_0(s) ] e^{-(s-\tau) \tilde{\mathcal{A}}_0(\tau)} \, ds \\
&= \int_{\tau}^t \tilde{L}_0(t, s) [a_{\epsilon}(\tau) - a_{\epsilon}(s)] \hat{A} \, e^{-(s-\tau) \tilde{\mathcal{A}}_0(\tau)} \, ds
- \int_{\tau}^t \tilde{L}_0(t, s) [a_0(\tau) - a_0(s)] \hat{A} \, e^{-(s-\tau) \tilde{\mathcal{A}}_0(\tau)} \, ds \\
&= [a_{\epsilon}(\tau) - a_0(\tau)] \int_{\tau}^t \tilde{L}_{0}(t, s) \hat{A} \, e^{-(s-\tau) \tilde{\mathcal{A}}_0(\tau)} \, ds
+ \int_{\tau}^t [a_0(s) - a_{\epsilon}(s)] \tilde{L}_{0}(t, s) \hat{A} \, e^{-(s-\tau) \tilde{\mathcal{A}}_0(\tau)} \, ds
\end{split}
\]
which leads to 
\[
\begin{split}
&\left\Vert \int_{\tau}^t \tilde{L}_{0}(t, s) [ \tilde{\mathcal{A}}_{\epsilon}(\tau) - \tilde{\mathcal{A}}_{\epsilon}(s) ] e^{-(s-\tau) \tilde{\mathcal{A}}_{0}(\tau)} \, ds \right.\\
& \hspace{5cm}\left.- \int_{\tau}^t \tilde{L}_{0}(t, s) [ \tilde{\mathcal{A}}_0(\tau) - \tilde{\mathcal{A}}_0(s) ] 
e^{-(s-\tau) \tilde{\mathcal{A}}_{0}(\tau)} \, ds \right\Vert_{\mathcal{L}(Y_0)} \\
&\leq \left\Vert [a_{\epsilon}(\tau) - a_0(\tau)] \int_{\tau}^t \tilde{L}_{0}(t, s) \hat{A} \, e^{-(s-\tau) \tilde{\mathcal{A}}_{0}(\tau)} \, ds \right\Vert_{\mathcal{L}(Y_0)} \\
&\quad + \left\Vert \int_{\tau}^t [a_0(s) - a_{\epsilon}(s)] \tilde{L}_{0}(t, s) \hat{A} \, e^{-(s-\tau) \tilde{\mathcal{A}}_{0}(\tau)} \, ds \right\Vert_{\mathcal{L}(Y_0)} \\
\end{split}
\]

\[
\begin{split}
&\leq\vert a_{\epsilon}(\tau) - a_0(\tau)\vert \left\Vert \int_{\tau}^t \tilde{L}_{0}(t, s) \hat{A} \, e^{-(s-\tau) \tilde{\mathcal{A}}_{0}(\tau)} \, ds \right\Vert_{\mathcal{L}(Y_0)} \\
&\quad + \int_{\tau}^t \vert a_0(s) - a_{\epsilon}(s)\vert \Vert \tilde{L}_{0}(t, s) \hat{A} \, e^{-(s-\tau) \tilde{\mathcal{A}}_{0}(\tau)} \Vert_{\mathcal{L}(Y_0)} \, ds \\
&\leq \Vert a_{\epsilon} - a_0 \Vert_{L^{\infty}(\mathbb{R})} \int_{\tau}^t \Vert \tilde{L}_{0}(t, s) \hat{A} \, e^{-(s-\tau) \tilde{\mathcal{A}}_{0}(\tau)} \Vert_{\mathcal{L}(Y_0)} \, ds \\
&\quad + \int_{\tau}^t \Vert a_0 - a_{\epsilon} \Vert_{L^{\infty}(\mathbb{R})} \Vert \tilde{L}_{0}(t, s) \hat{A} \, e^{-(s-\tau) \tilde{\mathcal{A}}_{0}(\tau)} \Vert_{\mathcal{L}(Y_0)} \, ds \\
&= 2 \Vert a_{\epsilon} - a_0 \Vert_{L^{\infty}(\mathbb{R})} \int_{\tau}^t \Vert \tilde{L}_{0}(t, s) \hat{A} \, e^{-(s-\tau) \tilde{\mathcal{A}}_{0}(\tau)} \Vert_{\mathcal{L}(Y_0)} \, ds 
\xrightarrow{\epsilon \rightarrow 0^+} 0,
\end{split}
\]
which proves that 
\begin{equation}\label{integral A goes to zero}
\begin{split}
&\left\Vert \int_{\tau}^t \tilde{L}_{0}(t, s) [ \tilde{\mathcal{A}}_{\epsilon}(\tau) - \tilde{\mathcal{A}}_{\epsilon}(s) ] e^{-(s-\tau) \tilde{\mathcal{A}}_{0}(\tau)} \, ds \right.\\ 
&\hspace{3cm} \left.- \int_{\tau}^t \tilde{L}_{0}(t, s) [ \tilde{\mathcal{A}}_0(\tau) - \tilde{\mathcal{A}}_0(s) ] 
e^{-(s-\tau) \tilde{\mathcal{A}}_{0}(\tau)} \, ds \right\Vert_{\mathcal{L}(Y_0)} \xrightarrow{\epsilon\rightarrow 0^{+}} 0.
\end{split}
\end{equation}

Next, in order to deal with item (B), we write
\[
\begin{split}
&\Vert (\tilde{L}_{\epsilon}(t, s) - \tilde{L}_0(t, s) ) [ \tilde{\mathcal{A}}_{\epsilon}(\tau) - \tilde{\mathcal{A}}_{\epsilon}(s) ] e^{-(s-\tau) \tilde{\mathcal{A}}_0(\tau)} 
\Vert_{\mathcal{L}(Y_0)} \\
&\leq \Vert \tilde{L}_{\epsilon}(t, s) - \tilde{L}_0(t, s) \Vert_{\mathcal{L}(Y_0)} 
\Vert [ \tilde{\mathcal{A}}_{\epsilon}(\tau) - \tilde{\mathcal{A}}_{\epsilon}(s) ] 
\tilde{\mathcal{A}}_{\epsilon}^{-1}(\tau) \Vert_{\mathcal{L}(Y_0)} \\
&\hspace{5cm} \times \Vert \tilde{\mathcal{A}}_{\epsilon}(\tau) \tilde{\mathcal{A}}_0^{-1}(\tau) 
\Vert_{\mathcal{L}(Y_0)} \Vert \tilde{\mathcal{A}}_0(\tau) e^{-(s-\tau) \tilde{\mathcal{A}}_0(\tau)} 
\Vert_{\mathcal{L}(Y_0)}
\end{split}
\]
and then using the estimates
\[
\Vert \tilde{\mathcal{A}}_0(\tau) e^{-(s-\tau) \tilde{\mathcal{A}}_0(\tau)} 
\Vert_{\mathcal{L}(Y_0)} \leq \tilde{C} (s-\tau)^{-1}, \, s-\tau > 0, \, \tau\in\mathbb{R},
\]
for some constant $\tilde{C} > 0,$
\begin{equation}\label{boundness Ae A0}
\Vert \tilde{\mathcal{A}}_{\epsilon}(\tau) \tilde{\mathcal{A}}_0^{-1}(\tau) 
\Vert_{\mathcal{L}(Y_0)} \leq 
1 + \Vert a_{\epsilon} - a_0 \Vert_{L^{\infty}(\mathbb{R})}, \, \tau\in\mathbb{R},
\end{equation}
and Proposition \ref{uniform holder A til}, it follows that

\[
\begin{aligned}
&\left\Vert
\int_{\tau}^t (\tilde{L}_{\epsilon}(t, s) - \tilde{L}_0(t, s) ) [ \tilde{\mathcal{A}}_{\epsilon}(\tau) - \tilde{\mathcal{A}}_{\epsilon}(s) ] e^{-(s-\tau) \tilde{\mathcal{A}}_0(\tau)} \, ds 
\right\Vert_{\mathcal{L}(Y_0)} \\
&\leq \int_{\tau}^t 
\Vert 
(\tilde{L}_{\epsilon}(t, s) - \tilde{L}_0(t, s) ) [\tilde{\mathcal{A}}_{\epsilon}(\tau) - \tilde{\mathcal{A}}_{\epsilon}(s)] e^{-(s-\tau) \tilde{\mathcal{A}}_0(\tau)} 
\Vert_{\mathcal{L}(Y_0)} \, ds \\
&\leq \int_{\tau}^t 
\Vert \tilde{L}_{\epsilon}(t, s) - \tilde{L}_0(t, s) \Vert_{\mathcal{L}(Y_0)} 
\Vert [\tilde{\mathcal{A}}_{\epsilon}(\tau) - \tilde{\mathcal{A}}_{\epsilon}(s)] 
\tilde{\mathcal{A}}_{\epsilon}^{-1}(\tau) \Vert_{\mathcal{L}(Y_0)} \\
&
\hspace{5cm} \times
\Vert \tilde{\mathcal{A}}_{\epsilon}(\tau) \tilde{\mathcal{A}}_0^{-1}(\tau) 
\Vert_{\mathcal{L}(Y_0)} 
\Vert \tilde{\mathcal{A}}_0(\tau) e^{-(s-\tau) \tilde{\mathcal{A}}_0(\tau)} 
\Vert_{\mathcal{L}(Y_0)} \, ds \\
&\leq \int_{\tau}^t 
\Vert \tilde{L}_{\epsilon}(t, s) - \tilde{L}_0(t, s) \Vert_{\mathcal{L}(Y_0)} 
C\vert \tau - s\vert^{\beta} 
\left( 1 + \Vert a_{\epsilon} - a_0 \Vert_{L^{\infty}(\mathbb{R})} \right) 
\tilde{C} (s-\tau)^{-1} \, ds \\
\end{aligned}
\]

\begin{equation}\label{conv linear process 04}
\begin{aligned}
&= \tilde{\tilde{C}} 
\left( 1 + \Vert a_{\epsilon} - a_0 \Vert_{L^{\infty}(\mathbb{R})} \right) 
\int_{\tau}^t 
\Vert \tilde{L}_{\epsilon}(t, s) - \tilde{L}_0(t, s) \Vert_{\mathcal{L}(Y_0)} 
\vert \tau - s\vert^{\beta} 
(s-\tau)^{-1} \, ds \\
&= \tilde{\tilde{C}} 
\left( 1 + \Vert a_{\epsilon} - a_0 \Vert_{L^{\infty}(\mathbb{R})} \right) 
\int_{\tau}^t 
\vert \tau - s\vert^{\beta - 1} 
\Vert \tilde{L}_{\epsilon}(t, s) - \tilde{L}_0(t, s) \Vert_{\mathcal{L}(Y_0)} \, ds.
\end{aligned}
\end{equation}

Now, for the integral on item (C), we write
\begin{equation}\label{conv linear process 06}
\begin{split}
&\left\Vert \int_{\tau}^t \tilde{L}_{\epsilon}(t, s) [\tilde{\mathcal{A}}_{\epsilon}(\tau)
- \tilde{\mathcal{A}}_{\epsilon}(s)] (e^{-(s-\tau) \tilde{\mathcal{A}}_{\epsilon}(\tau)} - e^{-(s-\tau) \tilde{\mathcal{A}}_0(\tau)} ) \, ds \right\Vert_{\mathcal{L}(Y_0)} \\
&\leq \int_{\tau}^t \Vert \tilde{L}_{\epsilon}(t, s) [\tilde{\mathcal{A}}_{\epsilon}(\tau)
- \tilde{\mathcal{A}}_{\epsilon}(s)] (e^{-(s-\tau) \tilde{\mathcal{A}}_{\epsilon}(\tau)} - e^{-(s-\tau) \tilde{\mathcal{A}}_0(\tau)} ) \Vert_{\mathcal{L}(Y_0)} \, ds \\
&\leq \int_{\tau}^t \Vert \tilde{L}_{\epsilon}(t, s) \Vert_{\mathcal{L}(Y_0)} 
\Vert \tilde{\mathcal{A}}_{\epsilon}(\tau) 
- \tilde{\mathcal{A}}_{\epsilon}(s) \Vert_{\mathcal{L}(Y_1, Y_0)} 
\Vert e^{-(s-\tau) \tilde{\mathcal{A}}_{\epsilon}(\tau)} - e^{-(s-\tau) \tilde{\mathcal{A}}_0(\tau)} \Vert_{\mathcal{L}(Y_0, Y_1)} \, ds
\end{split}
\end{equation}

Now, according to \cite[Theorem 2.2]{Nascimento}, the process $\{\tilde{L}_{\epsilon}(t, \tau)\colon t\geq\tau\in\mathbb{R}\}$ satisfies
\[
\Vert \tilde{L}_{\epsilon}(t, \tau) \Vert_{\mathcal{L}(Y_0)} \leq M, \quad 
t\geq\tau\in\mathbb{R},
\]
with constant $M>0$ being independent of the time $t\in\mathbb{R}$ and also of the parameter $\epsilon$. With this, we can combine \eqref{conv linear process 06}, Lemma \ref{conv semigroups in Y0 Y1}, Proposition \ref{uniform holder A til}, and the Dominated Convergence Theorem in order to obtain
\[
\begin{split}
&\left\Vert \int_{\tau}^t \tilde{L}_{\epsilon}(t, s) [\tilde{\mathcal{A}}_{\epsilon}(\tau)
- \tilde{\mathcal{A}}_{\epsilon}(s)] (e^{-(s-\tau) \tilde{\mathcal{A}}_{\epsilon}(\tau)} - e^{-(s-\tau) \tilde{\mathcal{A}}_0(\tau)} ) \, ds \right\Vert_{\mathcal{L}(Y_0)} \\
&\leq \tilde{M} \int_{\tau}^t \vert\tau -s\vert^{\beta} 
\Vert e^{-(s-\tau) \tilde{\mathcal{A}}_{\epsilon}(\tau)} - 
e^{-(s-\tau) \tilde{\mathcal{A}}_0(\tau)} \Vert_{\mathcal{L}(Y_0, Y_1)} \, ds
\xrightarrow{\epsilon\rightarrow 0^{+}} 0,
\end{split}
\]
which yields
\begin{equation}\label{conv linear process 10}
\left\Vert \int_{\tau}^t \tilde{L}_{\epsilon}(t, s) 
[\tilde{\mathcal{A}}_{\epsilon}(\tau) 
- \tilde{\mathcal{A}}_{\epsilon}(s)] (e^{-(s-\tau) \tilde{\mathcal{A}}_{\epsilon}(\tau)} - e^{-(s-\tau) \tilde{\mathcal{A}}_0(\tau)} ) \, ds \right\Vert_{\mathcal{L}(Y_0)} 
\xrightarrow{\epsilon\rightarrow 0^{+}} 0.
\end{equation}

Finally, combining \eqref{conv linear process 01}, \eqref{conv linear process 03} and \eqref{conv linear process 04}, we can write
\[
\small
\begin{split}
&\Vert \tilde{L}_{\epsilon}(t, \tau) - \tilde{L}_0(t, \tau) \Vert_{\mathcal{L}(Y_0)} 
\leq \Vert e^{-(t-\tau) \tilde{\mathcal{A}}_{\epsilon}(\tau)} - e^{-(t-\tau) \tilde{\mathcal{A}}_0(\tau)} \Vert_{\mathcal{L}(Y_0)} \\
&+ \left\Vert 
\int_{\tau}^t \tilde{L}_0(t, s) [ \tilde{\mathcal{A}}_{\epsilon}(\tau) - \tilde{\mathcal{A}}_{\epsilon}(s) ] e^{-(s-\tau) \tilde{\mathcal{A}}_0(\tau)} \, ds 
- \int_{\tau}^t \tilde{L}_0(t, s) [ \tilde{\mathcal{A}}_0(\tau) - \tilde{\mathcal{A}}_0(s) ] e^{-(s-\tau) \tilde{\mathcal{A}}_0(\tau)} \, ds 
\right\Vert_{\mathcal{L}(Y_0)} \\
&+ \tilde{\tilde{C}} 
\left( 1 + \Vert a_{\epsilon} - a_0 \Vert_{L^{\infty}(\mathbb{R})} \right) 
\int_{\tau}^t 
\vert \tau - s\vert^{\beta - 1} 
\Vert \tilde{L}_{\epsilon}(t, s) - \tilde{L}_0(t, s) \Vert_{\mathcal{L}(Y_0)} \, ds \\
&+ \left\Vert 
\int_{\tau}^t \tilde{L}_{\epsilon}(t, s) [ \tilde{\mathcal{A}}_{\epsilon}(\tau) - \tilde{\mathcal{A}}_{\epsilon}(s)] (e^{-(s-\tau) \tilde{\mathcal{A}}_{\epsilon}(\tau)} - e^{-(s-\tau) \tilde{\mathcal{A}}_0(\tau)} ) \, ds 
\right\Vert_{\mathcal{L}(Y_0)}
\end{split}
\]
and then, using \eqref{conv linear process 02}, \eqref{integral A goes to zero}, \eqref{conv linear process 10} and Generalized Gronwall inequality (see \cite[Theorem 1.26]{Yagi}) we have
\[
\Vert \tilde{L}_{\epsilon}(t, \tau) - \tilde{L}_0(t, \tau) \Vert_{\mathcal{L}(Y_0)}  \rightarrow 0, \quad \text{as } \epsilon\rightarrow 0^{+},
\]
which completes the proof.
\end{proof}

\begin{corollary}\label{all hyperbolic}
All the equilibrium points in $\mathcal{E}^{\ast}$ are hyperbolic for the family of processes 
$
\{ \tilde{L}_{\epsilon}(t, \tau) \colon t\geq \tau\in\mathbb{R}\},
$ 
for every $\epsilon\in [0, \epsilon_0],$ for some $\epsilon_0 >0$ sufficiently small.
\end{corollary}

\begin{proof}
Indeed, by our assumptions, all equilibrium points in $\mathcal{E}^{\ast}$ are hyperbolic, in the sense of Definition \ref{def hyperbolic sol}, 
for the limit problem ($\epsilon = 0$) associated to the system \eqref{edp01}-\eqref{cond01}, which means that the process 
$
\{ \tilde{L}_0(t, \tau) \colon t\geq \tau\in\mathbb{R}\}
$ 
has an exponential dichotomy in $Y_0$, in the sense of Definition \ref{def exp dicho}. 
On the other hand, according to \cite[Theorem 2.2]{Nascimento}, we have 
\[
\Vert \tilde{L}_0(t, \tau) \Vert_{\mathcal{L}(Y_0)} \leq M, \quad 
t\geq\tau\in\mathbb{R},
\]
with $M>0$ being a constant. 
In particular, it follows that 
\[
\sup\{ \Vert \tilde{L}_0(t, \tau) \Vert_{\mathcal{L}(Y_0)} \colon 
0 \leq t - \tau \leq 1 \} < \infty.
\]

Now, with Proposition \ref{convergence processes linearization} in hands, and applying \cite[Theorem 7.6.10]{Henry}, it follows that, for some $\epsilon_0 >0$ sufficiently small, the process 
$
\{ \tilde{L}_{\epsilon}(t, \tau) \colon t\geq \tau\in\mathbb{R}\}
$ 
has an exponential dichotomy in $Y_0$, for every $\epsilon\in [0, \epsilon_0]$. 
Consequently, every point in the set of equilibria $\mathcal{E}^{\ast}$ is also hyperbolic for the process 
$
\{ \tilde{L}_{\epsilon}(t, \tau) \colon t\geq \tau\in\mathbb{R}\},
$ 
for every $\epsilon\in [0, \epsilon_0]$.
\end{proof}

\begin{theorem}\label{lower PA}
The family of pullback attractors 
$
\{ \mathbb{A}_{\epsilon}(t)\colon t\in\mathbb{R}\}
$ 
associated to the problem \eqref{edp01}-\eqref{cond01} is lower semicontinuous at $\epsilon = 0$, that is, for each $t\in\mathbb{R}$,
\[
{\rm d_H}(\mathbb{A}_0(t), \mathbb{A}_{\epsilon}(t)) \rightarrow 0 
\quad \mbox{as} \quad \epsilon\rightarrow 0^{+}.
\]
\end{theorem}

\begin{proof}
Applying Lemma \ref{Aux applic L 0} we have the uniform continuity, with respect to the parameter $\epsilon$, of the nonlinear evolution processes $\{S_{\epsilon}(t, \tau)\colon t\geq\tau\in\mathbb{R}\}$ associated to the problem \eqref{edp01}-\eqref{cond01}. 
The compactness of the closure of the union of pullback attractors is a direct consequence of the regularity result stated in Theorem \ref{regularity pullback attractor}. 
Moreover, in Theorem \ref{structure limit PA}, we obtained the gradient-like structure of the limit pullback attractor 
$\{\mathbb{A}_0(t)\colon t\in\mathbb{R}\}$ 
associated to the problem \eqref{edp01}-\eqref{cond01} with $\epsilon = 0$. 
It remains to verify the two conditions in Theorem \ref{lower paper Felipe}. But, as noted in \cite{Rivero1}, they are just consequences of the stability of the hyperbolic equilibria under perturbation, see also 
\cite[Theorem 6.1]{Carvalho and Langa} and \cite[Theorem 2.3]{Carvalho Ergodic}. 
Now, Corollary \ref{all hyperbolic} ensures that every equilibrium point in $\mathcal{E}^{\ast}$ is also hyperbolic with respect to the perturbed problem \eqref{edp01}-\eqref{cond01}, for every $\epsilon\in [0, \epsilon_0],$ for some $\epsilon_0 >0$ sufficiently small. Hence, we can apply Theorem \ref{lower paper Felipe}, which leads to the lower semicontinuity at $\epsilon =0$ of the family of pullback attractors 
$
\{ \mathbb{A}_{\epsilon}(t)\colon t\in\mathbb{R}\}
$ 
associated to the problem \eqref{edp01}-\eqref{cond01}.
\end{proof}

\begin{corollary}
The family of pullback attractors 
$
\{ \mathbb{A}_{\epsilon}(t)\colon t\in\mathbb{R}\},
$ 
associated to the non-autonomous problem \eqref{edp01}-\eqref{cond01}, is continuous at $\epsilon = 0$.
\end{corollary}

\begin{proof}
This result is a direct consequence of Theorems \ref{upper PA} and \ref{lower PA}.
\end{proof}

%
%
%
%
%
%

\begin{appendices}
\section{Appendix}\label{appendix A}

\subsection{Evolution processes and pullback attractors}

Let $(Z, d)$ be a metric space and $\N = \{1,2,\ldots\}$ be the set of all natural numbers.  An \textit{evolution process} acting in $Z$ is a two-parameter family $\{ S(t, \tau): t\geq\tau\in\mathbb{R} \}$ of maps from $Z$ into itself such that:
\begin{enumerate}
	\item[$\bullet$] $S(t, t) = I$ for all $t \in\mathbb{R}$ ($I$ is the identity operator in $Z$),
	\item[$\bullet$] $S(t, \tau) = S(t, s)S(s, \tau)$ for all $t\geq s\geq\tau$, and
	\item[$\bullet$] the map $\{ (t, \tau)\in \mathbb{R}^2: t\geq\tau \} \times Z \ni (t, \tau, x) \mapsto S(t, \tau)x \in Z$ is continuous.
\end{enumerate}

For a given set  $B \subset Z$, 
$S(t, \tau)B = \{S(t, \tau)x \colon  x \in B\}$
is the image of $B$ under $\{ S(t, \tau)\colon  t\geq\tau \in\mathbb{R}\}$. If $Z$ is a Banach space and $\{ S(t, \tau)\colon  t\geq\tau\in\mathbb{R} \} \subset \mathcal{L}(Z)$, we refer to this process as a \textit{linear evolution process}.

The {\it Hausdorff semidistance} between two nonempty subsets $A$ and $B$ of $Z$ is given by
$$
{\rm d_H}(A, B) = \sup\limits_{a\in A} \inf\limits_{b\in B} d(a, b)
$$
and the \textit{symmetric Hausdorff distance} is given by
\[
{\rm dist_H}(A, B) = \max \{ {\rm d_H}(A, B), {\rm d_H}(B, A) \}.
\]

\begin{definition} \rm
A family $\{ \mathbb{A}(t) \colon t\in\mathbb{R} \}$ of compact subsets of $Z$ is a \textit{pullback attractor} for an evolution process $\{ S(t, \tau)\colon  t\geq\tau \in\mathbb{R} \}$ if the following conditions hold:
\begin{enumerate}
\item[$(i)$] $\{ \mathbb{A}(t)\colon  t\in\mathbb{R} \}$ is invariant, that is, $S(t, \tau) \mathbb{A}(\tau) = \mathbb{A}(t)$ for all $t\geq\tau$,

\item[$(ii)$] $\{ \mathbb{A}(t)\colon  t\in\mathbb{R} \}$ pullback attracts bounded subsets of $Z$, that is,
\[
\lim\limits_{\tau \to -\infty} {\rm d_H} (S(t, \tau)B, \mathbb{A} (t)) = 0
\]
for every $t\in\mathbb{R}$ and every bounded subset $B$ of $Z$,  and

\item[$(iii)$] $\{ \mathbb{A}(t)\colon  t\in\mathbb{R} \}$ is the minimal family of closed sets satisfying property $(ii)$.
\end{enumerate}
\end{definition}

A general result concerning the existence of pullback attractors states that, if an evolution process is pullback strongly bounded dissipative and pullback asymptotically compact, then it has a pullback attractor (see \cite[Theorem 2.23]{Livro Alexandre}). The theory of pullback attractors for nonlinear evolution processes is covered with details in \cite{Rivero2}, \cite{Livro Alexandre} and \cite{Chepyzhov and Vishik}.

\begin{definition} \rm
A \textit{global solution} for an evolution process $\{ S(t, \tau)\colon  t\geq\tau \in\mathbb{R} \}$ in $Z$ is a function $\xi\colon\mathbb{R}\to Z$ such that $S(t, \tau) \xi(\tau) = \xi(t)$ for all $t\geq\tau$.
\end{definition}

\begin{definition} \rm
An \textit{equilibrium point} $\xi^{\ast}\in Z$ for an evolution process $\{ S(t, \tau)\colon  t\geq\tau \in\mathbb{R} \}$ in $Z$ is a time independent global solution, that is,   $S(t, \tau)\xi^{\ast} = \xi^{\ast}$ for all $t\geq\tau$.
\end{definition}

\begin{definition}\rm 
Let $\{A_{\lambda}\}_{\lambda\in\Lambda}$ be a family of subsets of $Z$ indexed on a metric space $\Lambda$. The family 
$\{A_{\lambda}\}_{\lambda\in\Lambda}$ is said to be:
\begin{enumerate}
\item[$(i)$] \textit{upper semicontinuous as} $\lambda\rightarrow\lambda_0$ (or at $\lambda_0$) if $\lim\limits_{\lambda\rightarrow\lambda_0} {\rm d_H}(A_{\lambda}, A_{\lambda_0}) = 0$;

\item[$(ii)$] \textit{lower semicontinuous as} $\lambda\rightarrow\lambda_0$ (or at $\lambda_0$) if $\lim\limits_{\lambda\rightarrow\lambda_0} {\rm d_H}(A_{\lambda_0}, A_{\lambda}) = 0$;

\item[$(iii)$] \textit{continuous as} $\lambda\rightarrow\lambda_0$ (or at $\lambda_0$)  if it is both upper and lower semicontinuous as $\lambda\rightarrow\lambda_0$.
\end{enumerate}
\end{definition}

In the next two definitions, we consider a linear evolution process $\{L(t, \tau)\colon t\geq\tau\in\mathbb{R}\}$ acting on a Banach space $X$.

\begin{definition}\label{def exp dichotomy} \rm
A linear evolution process $\{L(t, \tau)\colon t\geq\tau\in\mathbb{R}\}$ in a Banach space $X$ admits  an \textit{exponential dichotomy} with projection 
$\{P(t)\colon t\in\mathbb{R}\} \subset \mathcal{L}(X)$, exponent $\omega > 0$ and constant $M \geq 1$, if:
\begin{enumerate}
\item[$(i)$] $P(t)L(t, \tau) = L(t, \tau)P(\tau)$ for all $t\geq\tau$;

\item[$(ii)$] the restriction of $L(t, \tau)$ to $R(P(\tau))$ is an isomorphism from $R(P(\tau))$ onto $R(P(t))$, where $R(P(t))$ denotes the range of the operator $P(t)$; and

\item[$(iii)$]  $\Vert L(t, \tau)(I - P(\tau)) \Vert_{\mathcal{L}(X)} \leq 
M e^{-\omega (t-\tau)}, \quad t\geq \tau,$

\noindent $\Vert L(t, \tau)P(\tau) \Vert_{\mathcal{L}(X)} \leq 
M e^{\omega (t-\tau)}, \quad t\leq \tau$.


\end{enumerate}
\end{definition}

\begin{remark} \rm In Definition \ref{def exp dichotomy}, item $(ii)$, we denote the inverse of $L(t, \tau)$ by 
$L(\tau, t)$.
\end{remark}

\begin{definition} \rm
A global solution $\xi\colon\mathbb{R}\to X$ for a linear evolution process 
$\{L(t, \tau)\colon t\geq\tau\in\mathbb{R}\}$ in a Banach space $X$ is said to be 
\textit{hyperbolic}, if $\{L(t, \tau)\colon t\geq\tau\in\mathbb{R}\}$ has an exponential dichotomy in the sense of Definition \ref{def exp dichotomy}.
\end{definition}

\end{appendices}

\begin{appendices}
\section{Appendix}\label{appendix B}

\subsection{An abstract result on lower semicontinuity of pullback attractors}

Let $X$ be a Banach space,  $\mathcal{B}\colon D(\mathcal{B}) \subset X\to X$ be the generator of a strongly continuous semigroup 
$
\{e^{\mathcal{B}t} \colon t\geq 0\} \subset \mathcal{L}(X),
$ $y_0\in X$ 
and consider the following abstract semilinear problem
\begin{equation}\label{semi p}
\begin{cases}
y_t = \mathcal{B}y + f_0(t, y), & t > \tau, \\
y(\tau) = y_0, & \tau \in \mathbb{R},
\end{cases}
\end{equation}
and its associated regular perturbation, given by
\begin{equation}\label{regular perturbation}
\begin{cases}
y_t = \mathcal{B}y + f_{\epsilon}(t, y), & t > \tau, \quad \epsilon\in [0, 1], \\
y(\tau) = y_0, & \tau \in \mathbb{R}.
\end{cases}
\end{equation}

If we assume that, for each $\epsilon\in [0, 1],$ the nonlinearity $f_{\epsilon} \colon \mathbb{R} \times X\to X$ is continuous and Lipschitz continuous in the second variable, uniformly on bounded subsets of $X$, then the problems \eqref{semi p} and \eqref{regular perturbation} are locally well-posed. Assuming that, for each $\tau \in \mathbb{R}$ and $y_0\in X$, the solution $t \mapsto S_{\epsilon}(t, \tau)y_0$ of \eqref{regular perturbation} is defined in $[\tau, \infty),$ we write
\[
S_{\epsilon}(t, \tau)y_0 = e^{\mathcal{B}(t-\tau)} y_0 + 
\int_{\tau}^t e^{\mathcal{B}(t-s)} f_{\epsilon}(s, S_{\epsilon}(s, \tau)y_0) \, ds.
\]
The family $\{S_{\epsilon}(t, \tau)\colon t\geq\tau\in\mathbb{R}\}, \epsilon\in [0, 1],$ consists of nonlinear evolution processes.

If $\xi_{\epsilon}^{\ast} \colon\mathbb{R}\to X, \epsilon\in [0, 1],$ is a global solution of \eqref{regular perturbation}, then we consider the linearization of \eqref{regular perturbation} around $\xi_{\epsilon}^{\ast}$ given by
\begin{equation}\label{abstract linearization}
\begin{cases}
z_t = \mathcal{B}z + (f_{\epsilon})_z(t, \xi_{\epsilon}^{\ast}(t))z, \\
z(\tau) = z_0,
\end{cases}
\end{equation}
and the evolution process $\{U_{\epsilon}(t, \tau)\colon t\geq\tau\in\mathbb{R}\}, \epsilon\in [0, 1],$ given by
\[
U_{\epsilon}(t, \tau)z_0 = e^{\mathcal{B}(t-\tau)}z_0 + 
\int_{\tau}^t e^{\mathcal{B}(t-s)} (f_{\epsilon})_z (s, \xi_{\epsilon}^{\ast}(s)) 
U_{\epsilon}(s, \tau)z_0 \, ds.
\]

\begin{definition}\label{def exp dicho} \rm
The linearized problem \eqref{abstract linearization} admits an 
\textit{exponential dichotomy} with exponent $\omega > 0$ and constant $M \geq 1$, if there exists a family of projections 
$
\{ \mathcal{Q}_{\epsilon}(t) \colon t\in\mathbb{R}\} \subset \mathcal{L}(X)
$ 
such that:
\begin{enumerate}
\item[$(i)$] $\mathcal{Q}_{\epsilon}(t) U_{\epsilon}(t, \tau) = U_{\epsilon}(t, \tau) \mathcal{Q}_{\epsilon}(\tau)$ for all $t\geq\tau$;

\item[$(ii)$] $U_{\epsilon}(t, \tau) \colon R(\mathcal{Q}_{\epsilon}(\tau)) \to R(\mathcal{Q}_{\epsilon}(t))$ is an isomorphism, with inverse 
\[
U_{\epsilon}(\tau, t) \colon R(\mathcal{Q}_{\epsilon}(t)) \to 
R(\mathcal{Q}_{\epsilon}(\tau)) 
\quad \mbox{for all} \quad t\geq\tau;
\]

\item[$(iii)$]  $\Vert U_{\epsilon}(t, \tau) (I - \mathcal{Q}_{\epsilon}(\tau)) 
\Vert_{\mathcal{L}(X)} \leq  M e^{-\omega (t-\tau)}, \quad \mbox{for} \quad t\geq \tau,$

\noindent $\Vert U_{\epsilon}(t, \tau) \mathcal{Q}_{\epsilon}(\tau) 
\Vert_{\mathcal{L}(X)} \leq 
M e^{\omega (t-\tau)}, \quad \mbox{for} \quad t\leq \tau.$
\end{enumerate}
\end{definition}

\begin{definition}\label{def hyperbolic sol} \rm
A global solution $\xi_{\epsilon}^{\ast} \colon\mathbb{R}\to X, \epsilon\in [0, 1],$ of \eqref{regular perturbation} is said to be \textit{hyperbolic} if the evolution process 
$
\{U_{\epsilon}(t, \tau)\colon t\geq\tau\in\mathbb{R}\}
$
associated to \eqref{abstract linearization} has an exponential dichotomy in the sense of Definition \ref{def exp dicho}.
\end{definition}

The following result concerns the lower semicontinuity of a family of pullback attractors and it was stated in \cite[Theorem 6.2]{Rivero1}. 
We also refer to \cite[Theorem 3.1]{Carvalho Ergodic} for a more general result.

\begin{theorem}\label{lower paper Felipe}
Let $\{S_{\epsilon}(t, \tau)\colon t\geq\tau\in\mathbb{R}\}, \epsilon\in [0, 1],$ be the family of nonlinear evolution processes associated to \eqref{regular perturbation} in $X$ and assume that, for $[\tau, t] \subset \mathbb{R},$
\[
\Vert S_{\epsilon}(t, \tau)x - S_0(t, \tau)x \Vert_X \rightarrow 0 
\quad \mbox{as} \quad \epsilon\rightarrow 0^{+},
\]
for any $x$ in a compact subset of $X$. Suppose that, for each $\epsilon\in [0, 1],$ the process $\{S_{\epsilon}(t, \tau)\colon t\geq\tau\in\mathbb{R}\}$ has a pullback attractor $\{\mathbb{A}_{\epsilon}(t) \colon t\in\mathbb{R}\}$ in $X$ such that, for $\epsilon_0 >0$ sufficiently small,
\[
\overline{
\bigcup\limits_{t\in\mathbb{R}} \bigcup\limits_{\epsilon\in [0, \epsilon_0]} 
\mathbb{A}_{\epsilon}(t)} 
\mbox{  is compact in  } X.
\]
Moreover, assume that there exists a finite collection $\mathcal{E}^{\ast} = \{e_1^{\ast}, \ldots, e_r^{\ast}\}$ of equilibrium points, all of them hyperbolic for the limit problem \eqref{semi p}, such that the limit pullback attractor $\{\mathbb{A}_0(t) \colon t\in\mathbb{R}\}$ is the union of their unstable manifolds, that is,
\[
\mathbb{A}_0(t) = \bigcup_{i = 1}^r W^u(e_i^{\ast})(t) \quad \mbox{for all} \quad 
t\in\mathbb{R},
\]
where $W^u(e_i^{\ast})(t) = \{\zeta \in X: \text{ there is a global solution } \xi\colon \R \to X  \text{ such that } \xi(t) = \zeta  \text{ and } \displaystyle\lim_{\tau \to -\infty} dist(\xi(\tau), e_i^{\ast}) =0\}$. Further, assume that for each $e_j^{\ast} \in\mathcal{E}^{\ast}$:
\begin{enumerate}
\item[$\bullet$] given $\delta >0$, there exists $\epsilon_{j,\delta} >0$ such that, for all $0 < \epsilon < \epsilon_{j,\delta}$ there is a global hyperbolic solution $\xi_{j,\epsilon}$ of \eqref{regular perturbation} that satisfies
\[
\sup\limits_{t\in\mathbb{R}} \Vert \xi_{j,\epsilon}(t) - e_j^{\ast} \Vert_X < \delta,
\]

\item[$\bullet$] the local unstable manifold of $\xi_{j,\epsilon}$ behaves continuously as $\epsilon\rightarrow 0^{+}$, that is,
\[
{\rm dist_H}\left( W_{0, {\rm loc}}^u (e_j^{\ast}), 
W_{\epsilon, {\rm loc}}^u (\xi_{j,\epsilon}) \right) \rightarrow 0 
\quad \mbox{as} \quad \epsilon\rightarrow 0^{+},
\]
\end{enumerate}
where ${\rm dist_H}(\cdot, \cdot)$ is the symmetric Hausdorff distance and $W_{{\rm loc}}^u (\cdot) = W^u(\cdot) \cap B_X(\cdot, \rho),$ with $\rho >0$.

Then the family 
$
\{\mathbb{A}_{\epsilon}(t) \colon t\in\mathbb{R}, \epsilon\in [0, \epsilon_0] \}
$ 
is lower semicontinuous at $\epsilon = 0$, that is,
\[
{\rm d_H}(\mathbb{A}_0(t), \mathbb{A}_{\epsilon}(t)) \rightarrow 0 
\quad \mbox{as} \quad \epsilon\rightarrow 0^{+}.
\]
\end{theorem}

\end{appendices}

\bibliographystyle{amsplain}

\end{document}